\documentclass[11pt]{article}
\usepackage[utf8]{inputenc}
\usepackage{geometry}
\usepackage{authblk}
\usepackage{lmodern}
\usepackage[T1]{fontenc}
\usepackage{amsfonts}
\usepackage{mathrsfs}
\usepackage{indentfirst}
\usepackage{amsmath}
\usepackage{amsthm}
\usepackage{enumerate}
\usepackage{enumitem}
\usepackage{amssymb}
\usepackage{relsize}
\usepackage{setspace}
\usepackage{mathtools}
\usepackage{abraces}
 \usepackage[all]{xy}
\usepackage{bbm}
\usepackage{graphicx}
\usepackage{tikz}
\usepackage{tikz-qtree}
\usetikzlibrary{shapes.geometric, arrows}
\usepackage{hyperref}
\usepackage{url}
\usepackage{pgfplots}
\pgfplotsset{compat=1.15}
\usetikzlibrary{arrows}

\numberwithin{equation}{section}

\DeclareMathOperator{\indicatrice}{\mathbbm{1}}
\DeclareMathOperator{\tv}{d_{\mathrm{TV}}}
\DeclareMathOperator{\argmin}{\mathrm{argmin}}

\newtheorem{theorem}{Theorem}[section]
\newtheorem{corollary}[theorem]{Corollary}
\newtheorem{lemma}[theorem]{Lemma}

\newtheorem{proposition}[theorem]{Proposition}
\newtheorem{claim}[theorem]{Claim}
\newtheorem*{claim*}{Claim}

\newtheorem{example}[theorem]{Example}
\newtheorem{def/prop}[theorem]{Definition/Proposition}

\newcommand{\toas}{\overset {\mathrm{a.s.}}{\longrightarrow}}
\theoremstyle{definition}
\newtheorem{remark}[theorem]{Remark}

\newcommand\bB{{\mathbb B}}

\newcommand\bE{{\mathbb E}}

\newcommand\bG{{\mathbb G}}

\newcommand\bM{{\mathbb M}}
\newcommand\bN{{\mathbb N}}

\newcommand\bP{{\mathbb P}}

\newcommand\bR{{\mathbb R}}
\newcommand\bS{{\mathbb S}}

\newcommand\cE{{\cal{E}}}
\newcommand\cF{{\cal{F}}}

\newcommand\cI{{\cal{I}}}

\newcommand\cU{{\cal{U}}}

\newcommand\cY{{\cal{Y}}}
\newcommand\cZ{{\cal{Z}}}

\newcommand{\wt}{\widetilde}
\newcommand{\sref}[2]{\hyperref[#2]{#1 \ref{#2}}}

\title{The temporal stochastic block model}

\author[1,3]{Sofiya Burova\thanks{
Research of S.B. supported by the Spanish Agencia Estatal de Investigaci\'on under project PID2022-138268NB-100 and Grant PID2020-113082GB-I00 funded by MICIU/AEI/10.13039/501100011033.}}
\author[2,3,4]{G\'abor Lugosi\thanks{
Research of G.L. supported by the Spanish Ministry of Economy and Competitiveness, Grant PGC2018-101643-B-I00 and FEDER, EU.}}
\author[1,5]{Guillem Perarnau\thanks{
Research of G.P. supported by the Spanish Agencia Estatal de Investigaci\'on under projects PID2020-113082GB-I00 and the Severo Ochoa and Mar\'ia de Maeztu Program for Centers and Units of Excellence in R\&D (CEX2020-001084-M).}}
\affil[1]{IMTECH and Departament de Matem\`atiques, Universitat Polit\`ecnica de Catalunya, Barcelona, Spain}
\affil[2]{ICREA, Pg. Llu\'is Companys 23, 08010 Barcelona, Spain}
\affil[3]{Department of Economics and Business, Pompeu Fabra University, Barcelona, Spain}
\affil[4]{Barcelona School of Economics, Barcelona, Spain}
\affil[5]{Centre de Recerca Matem\`atica, Barcelona, Spain}

\date{\today}

\begin{document}

\maketitle

\begin{abstract}
Motivated by the need to understand infection spreading in inhomogeneous populations,
we consider a \emph{temporal} version of the stochastic block model, where each edge is equipped with a random label, interpreted as a timestamp.
We study the size and structure of reachable sets via increasing paths (that is, paths whose edge timestamps are strictly increasing) 
where the connections per node are of the order of $\log n$.
We prove tight results for the first order asymptotics of the size of the reachable set from a typical vertex, and the proportions of reached vertices in each community.
We also determine the typical length of the shortest and longest increasing paths with one or two fixed endpoints. In the regime where the reachable set $S$ of a typical vertex is of sublinear size, we describe the structure of the subgraph induced by $S$ by identifying an almost spanning tree (i.e., spanning all but $o(|S|)$ vertices), that is distributed
as a weighted random recursive tree, whose height, profile and degree sequence are known. In particular, when $|S|\ll \sqrt{n}$, the subgraph induced by $S$ is itself distributed as this well-understood weighted random recursive tree. 


\end{abstract}

\newpage
\section{Introduction}
\label{introduction}

\subsection{Motivation, notation and context}
In studying the spread of a viral epidemic or the spread of a rumor in a network of interacting individuals, the time when contacts occur is of crucial importance, since an encounter leads to an infection only if one of the two individuals involved was already infected at the time of the encounter. This consideration led to the model of \emph{temporal graphs}, in which edges are equipped with timestamps.
Indeed, temporal graphs have attracted considerable attention.  

Given a graph $G=(V,E)$, we construct the temporal graph $(G,\pi)$ by assigning to each edge $e \in E$ a unique label or \emph{timestamp}, given by $\pi(e)$. In the literature, as in our case, $\pi$ is often defined as a random permutation of the edges of $G$, independent of $G$. These timestamps may be interpreted as arrival times of the edges, inducing a natural dynamic construction of the graph in discretisized time, each edge $e$ arriving at \emph{time} $\pi(e)$.

Intuitively, we may interpret $(G,\pi)$ as a propagation process such as rumor spreading, infections, etc., where each edge $e$ represents an interaction between its endpoints at time $\pi(e)$ during which the transmission of information or disease occurs. Take $v \in V$ to be the \emph{origin} (or patient zero), that is the first node to get infected. Then, a vertex $u \in V$ becomes infected if and only if there is an \emph{increasing path} from the origin $v$ to $u$. We say that a path $(w_1,\ldots,w_k)$, $w_i \in V$, $w_i\neq w_j$, $w_iw_{i+1}\in E$ for all $i,j \in [k]$, is \emph{increasing} if the sequence of labels $(\pi(w_i w_{i+1}))_{i \in [k-1]}$ is increasing. The length of a path is defined as the number of edges in it. Furthermore, denote by $B_{(G,\pi)}(v)$ the set of vertices that can be reached from $v$ via increasing paths. The reachable set represents the infection of origin $v$. 

The notion of connectivity in this model is restricted to the existence of increasing paths between vertices. Even when the underlying graph is undirected, such paths are neither symmetric nor transitive: the existence of temporal paths from $u$ to $v$ and from $v$ to $w$ does not imply the existence of a temporal path from $u$ to $w$. As a consequence, reachability properties differ significantly from those in static graphs.
These differences have important algorithmic implications. For instance, it was established in \cite{KKK02} that deciding whether there exist $k$ vertex-disjoint temporal paths between two vertices is NP-complete, whereas the corresponding problem in static graphs is polynomial-time solvable. 
These quantities have been extensively studied in the case when the underlying graph is an Erd\H{o}s-Rényi random graph (\cite{AFST20, BCCKRRZ23, CRRZ24, BKL24,AtDeLu25a,AtDeLu25d}).

The main objects of interest in this paper are longest increasing paths (with 0, 1 or 2 fixed endpoints), shortest paths between two fixed vertices, and the reachable set $B_{(G,\pi)}(v)$ with origin at a fixed vertex $v$, as well as the structure of the induced subgraph $G[B_{(G,\pi)}(v)]$. These quantities exhaustively describe the size, speed and depth of propagation processes on a given network. In particular, they quantify how far and how efficiently information or influence can spread from a given source under temporal constraints.

We are interested in the asymptotic behavior of temporal graphs $(G,\pi)$ where $G$ is a random graph on $n$ vertices. Throughout the paper, we use standard asymptotic notation $O, o, \Theta, \Omega, \sim$, where all limits are taken as $n\to \infty$. Whenever indexed by $p$, they describe the limiting behavior of random variables with respect to convergence in probability, for example, $X_n \sim_p a_n$ means $X_n/a_n \to 1$ in probability, as $n\to \infty.$ We say that a sequence of events $(E_n)_{n\in \mathbb{N}}$ occurs \emph{with high probability} (whp for brevity) if $\lim_{n\to +\infty}\mathbb{P}_n[E_n] = 1$.

It is known that the temporal Erd\H os-R\'enyi random graph $(\bG(n,p
),\pi)$ undergoes a phase transition at $p= c\log n /n$, where $c \in \bR^+$. In particular, Casteigts et al. showed in \cite{CRRZ24} that the thresholds for the properties that a typical pair of vertices is connected, a typical vertex can reach all other vertices, and any pair of vertices is connected are, respectively, $\log n/n, 2\log n/n$ and $3\log n/n$. Broutin, Kam\v  cev and Lugosi, \cite{BKL24}, extended these results providing tight bounds for the longest and shortest increasing paths in each regime, thus completing the study of the temporal diameter of $\bG(n,p)$. Furthermore, they show that, when $c<1$, 
the reachable set from a typical vertex 
is of size roughly $n^c$. This is achieved through first-moment arguments and by embedding a \emph{uniform random recursive tree} into the induced graph by the reachable set. This is a well-understood random object (see, e.g., \cite{P94, J05, DL95, AL18, D87}), obtained by repeatedly attaching a leaf to a uniform random vertex of the existing tree. When $c>1$,  the reachable set is of size $(1-o(1))n$, with high probability.

In this paper, we take the underlying graph $G$ to be distributed as the Stochastic Block Model (SBM). This is a well-studied random graph model (see \cite{A18} for a detailed summary) due its simplicity and vast applications across a wide range of fields, such as social network analysis, biology, information retrieval, etc. The stochastic block model is used to understand and analyze the structure of complex networks by partitioning nodes into distinct communities or blocks based on connection probabilities. 

The temporal stochastic block model is particularly useful in settings where interactions occur over time and their intensity depends on the types of the individuals involved. For instance, in epidemic models, communities may represent age groups, vaccination statuses, or other demographic classes, with different transmission rates within and between groups. 
This inhomogeneity is not captured by the temporal Erd\H os-R\' enyi model, where all pairs of vertices interact with the same probability. Incorporating community structure makes the model closer to real-world propagation processes.

In the setting, where all connection probabilities are of the same asymptotic order and all communities have linear size, our results describe the size of the infection, the proportions of infected individuals in each community, and the length of possible transmission chains. In particular, we observe that the first-order asymptotic behavior of the infection does not depend on the choice of the type of the initially infected vertex. Thus, the infection evolves in the same way regardless of the type of the origin. Our results also extend to asymmetric between-block transmission probabilities, where interactions between distinct communities may depend on direction, while interactions within a community remain undirected. For example, individuals who adopt protective measures (e.g., wear a mask in the case of an airborne disease) may be less likely to infect or be infected by others than vice versa. We discuss this extension, as well as other variations of the model, in Section \ref{discussion}.

\subsection{Main results}
In this section we state the main results of this paper. The main objects of interest, as mentioned in the previous subsection, are the size and structure of an infection process of fixed origin. We start with a formal definition of the model and its parameters.

\paragraph{Formal definition of the model.}Let $n \in \mathbb{N}$ and $k< n$ be positive integers, where $k$ indicates the number of blocks and $n$ the number of vertices. Furthermore, let $\boldsymbol{a} = (a_1,\ldots,a_k) \in (0,1]^k$ be a positive vector such that $\sum_{i=1}^k a_i = 1$, and let $P=P(n) \in [0,1]^{k\times k}$ be a symmetric $k\times k$ matrix. The stochastic block model of parameters $n,k,\boldsymbol{a},P$, denoted by $\mathbb{S}\mathbb{B}\mathbb{M}(n,k,\boldsymbol{a},P)$, is defined as follows. Take vertex set $V=[n]$ and partition it into $k$ blocks (or communities) $V_1,\ldots,V_k$ such that, for all $i \in [k]$, the size of $V_i$ is either $\lfloor a_i n\rfloor$ or $\lceil a_in \rceil$, where the choice is made arbitrarily in such a way that $\sum_{i\in[k]}|V_i| = n$. We often simplify the notation and write $|V_i| = a_i n$, as it has no impact on our results. For all $i,j \in [k]$, $u \in V_i$, $v \in V_j$, the edge $uv$ belongs to $E$ with probability $P_{ij}$, and edges are independent. Notice that if $k=1$ or if all entries of $P$ are identical, then we obtain $\bG(n,p)$. Let $(G,\pi)$ be the Temporal Stochastic Block Model (TSBM), where $G \sim \mathbb{S}\mathbb{B}\mathbb{M}(n,k,\boldsymbol{a},P)$ and $(\pi(e))_{e\in E(G)}$ is a permutation of $[|E(G)|]$ chosen uniformly at random. We consider $k$ to be a fixed constant as $n\to \infty$.\\

Let $G$ be a $\bS\bB\bM(n,k,\boldsymbol{a},P)$ and $\pi$ be a uniform random permutation of the edges of $G$. Equivalently, since only the \emph{ordering} of the edges matters in the analysis of increasing paths, one may assign to each unordered pair of vertices $uv$ an exponential random variable $\pi^\prime(uv)$ of parameter 1 (see \cite{CRRZ24}), thus equipping non-edges with time stamps as well. Observe that, through simple couplings of independent Bernoulli random variables, one may show that 
$$E(G_{\min}) \subseteq E(G) \subseteq E(G_{\max}), $$
where $G_{\min}$ and $G_{\max}$ are Erd\H os-R\'enyi random graphs on $n$ vertices, of parameters $p_{\min} \coloneqq \min_{i,j\in [k]}P_{ij}$ and $p_{\max}\coloneqq \max_{i,j\in [k]}P_{ij}$, resp. Hence, for any fixed $v \in V(G)$,
\begin{equation}
\label{obs_Gnp}
B_{(G_{\min},\pi^\prime)}(v) \subseteq B_{(G,\pi^\prime)}(v) \subseteq B_{(G_{\max},\pi^\prime)}(v).
\end{equation}
The results on the temporal Erd\H os-R\'enyi random graphs from \cite{BKL24} suggest that the regime $p\sim \log n/n$ is critical for the size of the reachable set. In light of \eqref{obs_Gnp}, we focus our analysis on the regime where all connection probabilities are either 0, or of the order $\log n/n$.\\

Throughout the paper, we assume the following mild conditions on the parameter of the model $P$:
\begin{enumerate}[label = (\roman*)]
	\item \label{ass1} $P$ is \emph{irreducible}, that is for all $i,j \in [k]$, there exists $m \in \mathbb{N}$ such that $(P^m)_{ij} > 0$.
	\item \label{ass2} The connection probabilities are either 0, or of the order $\log n/n$, that is, there exist non-negative constants $(\lambda_{ij})_{i,j\in [k]}$ such that $P_{ij} = \lambda_{ij}\log n/n$, for all $i,j \in [k]$.
\end{enumerate}
Notice that, the edge probabilities are \emph{not} necessarily non-zero. Hence, our results remain true, for example, in the case of bipartite random graphs.

Condition $(i)$ simply ensures that the graph is connected with non-zero probability. Condition $(ii)$ imposes all connection probabilities to have the same order with respect to $n$, which is a natural parametrization of stochastic block models. The choice for the regime $\log n/n$ is motivated by the observation in \eqref{obs_Gnp}.  

In order to state our results, we need to introduce the $k\times k$ matrix $\Lambda$, defined by $\Lambda_{ij} \coloneqq a_i \lambda_{ij}$ for all $i,j\in [k]$. This matrix may be interpreted as the expected proportion of vertices of type $i$ in the neighborhood of a vertex $u$ of type $j$, since
$$\bE[\#\{uw \in E : w\in V_i \}]= \Lambda_{ij} \log n.$$
Observe that, unlike the matrix $P$, $\Lambda$ is generally not symmetric. 
We apply a generalized version of the Perron-Frobenius theorem for non-negative matrices (see Chapter 7, \cite{M23}) to $\Lambda$ to deduce that its largest eigenvalue $\theta$ is unique, positive, and its associated eigenvector $x_\theta$ is comprised of non-negative entries, where we take $x_\theta$ satisfying $\boldsymbol{1}^T\cdot \Lambda \cdot x_\theta = 1$, with $\boldsymbol{1} = (1,\ldots,1)^T \in \mathbb{R}^k$. Note that this normalization implies $\|\theta x_{\theta}\|_1=1$. The main result of this paper shows that the size of the reachable set from a fixed vertex undergoes a phase transition at $\theta = 1$.

\begin{theorem}[Reachability from a vertex]\label{thm_reachability} Under assumptions $\ref{ass1}$ and \ref{ass2}, for any $v\in V$,
\begin{enumerate}[label=(\roman*)]
	\item \label{supercritical_reachability} if $\theta>1$, then $|B_{(G,\pi)}(v)| \sim_p n$;
	\item \label{subcritical_reachability} if $\theta<1$, then, for all $\epsilon > 0$, whp $$n^{(1-\epsilon)\theta} \leq |B_{(G,\pi)}(v)| \leq n^{(1+\epsilon)\theta}.$$ Moreover, for every $i \in [k]$, the proportion of vertices in $V_i$ reached by $v$ satisfies
	$$ \frac{|B_{(G,\pi)}(v)\cap V_i|}{|B_{(G,\pi)}(v)|} \sim_p (\theta x_\theta)_i.$$
\end{enumerate}
\end{theorem}

Theorem \ref{thm_reachability} provides a description of the reachable set of a fixed vertex in temporal stochastic block models, identifying a phase transition at $\theta = 1$ and determining the first-order asymptotics of its size and composition in both regimes. It has a natural interpretation in models of epidemic or information spreading on networks with community structure: the reachable set corresponds to the set of individuals infected from a given source. The result thus characterizes the typical size of an outbreak and identifies the threshold at which it becomes widespread. It also determines the asymptotic proportions of infected individuals within each community, which is particularly relevant in applications where one seeks to understand how different groups, such as age or gender classes, are affected. \\
\indent The main difficulty lies in the strong dependencies induced by the temporal constraint and the inhomogeneity of the connection probabilities. This prevents a direct application of classical branching process or exploration arguments. In particular, consider the natural recursive construction of $B_{(G,\pi)}(v)$, where we start with $\{v\}$ and then recursively attach a new vertex via the edge of minimal label among those exceeding the last one added (this is essentially building the \emph{first foremost tree}, see \cite{CRRZ24}). In the case of the temporal stochastic block model, this process heavily depends on the past, namely on the labels of previously added edges, as well as on the number of vertices of each type already added to the tree. Our approach relies on a coupling with auxiliary random structures where the edge labels posses the memoryless property, and a refined analysis of the aforementioned process, where the dependency on the previously added vertices' types is controlled via a P\'olya urn process. 

\begin{remark} When $\theta > 1$, the proportion of vertices of type $i$ reached by $v$ is with high probability $a_i$. Hence, this proportion also undergoes a phase transition from $(\theta x_\theta)_i$ when $\theta < 1$, to $a_i$ as $\theta > 1$. In general, $a_i \neq (\theta x_\theta)_i$, as can be seen in Example \ref{example_explicit}, where we compute explicit expressions for $\theta$ and $x_\theta$ in a handful of special cases. 
\end{remark}

In the proof of Theorem \ref{thm_reachability}, we show an intermediate result on the structure of the induced graph by $B_{(G,\pi)}(v)$. Before stating it, we need to define a rooted random tree model called a \emph{weighted random recursive tree}, introduced in \cite{BV06} in 2006, and studied particularly for its connection to preferential attachment trees \cite{BV06,HI17,MUB18,S18,ELO23}. Known results include tight asymptotic bounds for most tree properties (degree sequence, height, profile; see \cite{S18}). The model is formally defined through a recursive construction procedure. For any sequence of non-negative real numbers $(w_n)_{n\geq 1}$ with $w_1 > 0$, define the weighted random recursive tree of weights $(w_n)_{n\geq 1}$ as follows: at step 1, the tree is a single vertex 1, which is the root throughout the process. At step $t+1$, add vertex $t+1$ by attaching it via an edge to the randomly chosen vertex $K_{t+1}\in [t]$, where
$$\bP[K_{t+1} = k] \propto w_k,  $$
for all $k \in [t]$. $K_{t+1}$ is independent of the evolution of the tree and, hence, the random variables $K_2,K_3,\ldots$ are mutually independent. 
Denote by $(b_i)_{i\in [k]}$ the canonical basis of $\bR^k$, and let $\omega_i = b_i\cdot \Lambda \cdot \boldsymbol{1}$ for all $i\in [k]$. Let $(\sigma(m))_{m\in \bN} $ be a sequence of random variables, such that for all $m\in \bN$,
$$\bP[\sigma(m)=i|(\sigma(s))_{s< m }] \propto  b_i \cdot \Lambda \cdot X(m),$$
    where $X(m)$ is a random vector of $\bR^k$, whose $i^{\textrm{th}}$ entry is given by
    $$X_i(m) \coloneqq \#\{s < m: \sigma(s) = i \}.$$

Fix $v \in [n]$. For $M\in \mathbb{N}$, let $T_M(\boldsymbol{W})$ be a weighted random recursive tree of size $M$, and weights given by the random vector $\boldsymbol{W}=(W_1,\ldots,W_M)$, defined by $W_1 = \omega_i$ if $v \in V_i$, and for all $m>1, i\in[k]$, $W_m = \omega_{\sigma(m)}$. In other words,
    $$\bP[W_m = \omega_i|W_1,\ldots,W_{m-1}] \propto b_i \cdot \Lambda \cdot X(m).$$
    


Crucially, the random weights $(W_m)_{m\geq 1}$ are independent of the evolution of the tree. Indeed, we can compute them before building the tree, and thus construct $T_M(\boldsymbol{W})$ using a deterministic realization $(w_1,\ldots,w_M)$ of $\boldsymbol{W}$. It is easy to check that all possible realizations of the sequence satisfy the assumptions from \cite{S18}, since they all belong to the finite set $\{\omega_i : i\in[k]\}$. Hence, all their results on the height, degree sequence and profile, apply to $T_M(\boldsymbol{W})$. We are now ready to state our result.


\begin{proposition}
\label{proposition_subcritical_structure}
    Suppose assumptions \ref{ass1} and \ref{ass2} hold. Fix a vertex $v\in[n]$. Then, if $\theta < 1$, with high probability there exist $M=(1-o(1))|B_{(G,\pi)}(v)|$ and an embedding 
    $$\varphi : T_M(\boldsymbol{W}) \to G[B_{(G,\pi)}(v)], $$
    such that $\varphi(1)=v$, and all root-to-vertex paths are increasing in $(G,\pi)$. In particular, when $\theta < 1/2$, $M=|B_{(G,\pi)}(v)|$ and $\varphi$ is an isomorphism, with high probability. 
\end{proposition}

Proposition \ref{proposition_subcritical_structure} provides a structural description of the subgraph induced by the reachable set in the subcritical regime. In particular, up to a negligible fraction of vertices, $G[B_{(G,\pi)}(v)]$ can be represented as a weighted random recursive tree with increasing root-to-vertex paths. This representation provides insight on how the infection evolves generation by generation, as well as how long the infection chains are in this process. These results are used to derive bounds on the length of extremal increasing paths. 

\begin{remark} In the special case of Erd\H os-R\'enyi random graphs, that is when all entries of $P$ are equal to $c\log n/n$, when $c<1/2$ it follows that with high probability $B_{(G,\pi)}(v)$ \emph{is} a tree distributed as a uniform random recursive tree, thus strengthening the result from \cite{BKL24}, that states that one may \emph{embed} a uniform random recursive tree of size \emph{only} $n^{(1-\epsilon)c}$.\end{remark}


The following result completes the study of the first order asymptotics of the size and composition of the reachable set of a fixed vertex in the temporal stochastic block model.

\begin{proposition}
\label{proposition_subcritical}
Suppose assumption \ref{ass1} holds, and suppose all connection probabilities are of order $\omega_n/n$. In the regime $\omega_n = o(\log n)$, if $\omega_n \to \infty$, then whp $$\log |B_{(G,\pi)}(v)| \sim_p \theta \omega_n.$$
Moreover, the induced subgraph $G[B_{(G,\pi)}(v)]$ is whp a tree isomorphic to a weighted random recursive tree of weights $\boldsymbol{W}$, and the proportions of vertices of each type $i \in [k]$, reached by $v$, satisfy $$ \frac{|B_{(G,\pi)}(v)\cap V_i|}{|B_{(G,\pi)}(v)|} \overset{\bP}{\longrightarrow} (\theta x_\theta)_i.$$
\end{proposition} 

Now that we have Theorem \ref{thm_reachability} and Proposition \ref{proposition_subcritical_structure} at our disposal, the arguments for establishing tight bounds on the longest and shortest increasing paths can be adapted from the Erd\H os–R\' enyi case (see \cite{BKL24}). We include them for completeness.

Let $\gamma_{\textrm{max}}$, $\gamma_{\textrm{max}}(v)$ and $\gamma_{\textrm{max}}(u,v)$ be the longest increasing path, the longest increasing path starting at $v$, and the longest increasing path starting at $u$ and ending in $v$ in $G$, respectively. Similarly, define $\gamma_{\min}(u,v)$ to be the shortest increasing path from $u$ to $v$, should it exist. 

\begin{theorem}\label{thm_longest_paths}
Suppose assumptions \ref{ass1} and \ref{ass2} hold. Fix $u,v \in V$. Then, 
\begin{enumerate}[label = (\roman*)]
    \item \label{path}$|\gamma_{\mathrm{max}}| \in (e\theta\log n,\alpha(\theta)\log n)$ whp, where $$\alpha(\theta)\coloneqq \inf\left\{x>0: \log\left(\frac{x}{e\theta}\right)=1 \right\}  ;$$
	\item \label{path_v}$|\gamma_{\mathrm{max}}(v)| \sim_p e\theta\log n$;
	\item \label{path_uv}if $\theta>1$, then $|\gamma_{\mathrm{max}}(u,v)| \sim_p \beta(\theta)\log n$ and $|\gamma_{\mathrm{min}}(u,v)| \sim_p \gamma(\theta)\log n,$, where
    $$\beta(\theta) \coloneqq \sup\left\{x>0:x\log\left(\frac{x}{e\theta}\right)=-1\right\},$$
    $$\gamma(\theta) \coloneqq \inf\left\{x>0:x\log\left(\frac{x}{e\theta}\right)=-1\right\}. $$
    However, if $\theta<1$, then whp there is no increasing path between $u$ and $v$. 
\end{enumerate}
\end{theorem}

Notice that the lower bound in \ref{path} follows directly from \ref{path_v}. We expect the upper bound in \ref{path} to be tight, mirroring the result on Erd\H os-R\'enyi random graphs. However, deriving such a matching lower bound would require substantially more intricate arguments, and we do not pursue it here.

The reader is referred to \cite{BKL24} for a detailed discussion of the quantities $\alpha(\theta)$, $\beta(\theta)$ and $\gamma(\theta)$, as well as for the intuition underlying these results in the Erd\H os–R\'enyi case, on which our arguments are based.

\begin{example}\label{example_explicit}
In general, we cannot compute the parameter $\theta$ explicitly. However, in a handful of special cases, we obtain the following expressions.
\begin{itemize}
	\item $k=2$: 
	\begin{equation*}
	\Lambda =
	\begin{pmatrix}
	a_1p_1 & a_1q \\
	a_2q & a_2 p_2
	\end{pmatrix},
	\end{equation*}
	where, to avoid clutter, we change notation $\lambda_{ii} = p_i$ and $\lambda_{12} = \lambda_{21} = q$. The largest eigenvalue of $\Lambda$ is
	$$\theta = \frac{1}{2}\left(a_1 p_1+a_2 p_2 + \sqrt{(a_1 p_1-a_2 p_2)^2+4a_1 a_2 q^2}\right),$$
	and its associated eigenvector is $x_\theta \propto (a_1q,\theta - a_1p_1)$, which is, generally, different than $\boldsymbol{a}$.

	\item $a_i = 1/k$ and $\lambda_{ij} = q+\indicatrice_{i=j}(p-q)$ where $p,q > 0$ and $p\neq q$: 
	\begin{equation*}
	\Lambda = \frac{1}{k}
	\begin{pmatrix}
	p & q & \ldots & q \\
	q & p &  & \vdots \\
	\vdots &  & \ddots\textrm{ } & q \\
	q & \ldots & q & p
	\end{pmatrix}.
	\end{equation*}
	In this case, the largest eigenvalue is known to be (see \cite{A18}, Section 2.4)
	$$ \theta = (p+(k-1)q)/k. $$
	The associated eigenvector is $x_\theta = \boldsymbol{a}$. 
\end{itemize}
\end{example}

\subsection{Outline of the paper}

The proofs of the upper bounds in Theorem \ref{thm_reachability} and Proposition \ref{proposition_subcritical}, as well as the upper bounds for longest paths and lower bounds for shortest paths in Theorem \ref{thm_longest_paths} consist of applying the first moment method, and can be found in Section \ref{section_first_moments}.

The core of the paper is devoted to the analysis of the subcritical regime. We prove Proposition \ref{proposition_subcritical_structure}, the lower bounds in Theorem \ref{thm_reachability}, and Proposition \ref{proposition_subcritical} in Section \ref{section_subcritical}. 
The argument relies on a coupling with an auxiliary temporal random multigraph with exponentially distributed edge labels, together with a constructive exploration procedure. In Section \ref{section:proof_key_thm}, we prove the desired results for the auxiliary model.

The supercritical case is handled by a layering argument, combining subcritical components to show that the reachable set is of linear size.
Finally, the bounds on increasing paths in Theorem \ref{thm_longest_paths} follow from known results on weighted random recursive trees together with standard layering arguments. Section \ref{discussion} discusses extensions of the model and related open problems.

\section{First moment method}
\label{section_first_moments}

Let $(G,\pi)$ be a TSBM where $G \sim \bS\bB\bM(n,k,\boldsymbol{a},P)$. Define $\sigma : V \to [k]$ to be the map that assigns to a vertex $v \in V$ the block to which it belongs, that is, if $v\in V_i$, then $\sigma(v)=i$. Recall that an increasing path of size $\ell \in \bN$ is a sequence of $\ell+1$ distinct vertices $(w_1,\ldots,w_{\ell+1})$ such that $\pi(w_iw_{i+1}) < \pi(w_{i+1}w_{i+2})$ for all $i \in [\ell-1]$. For $\ell \in \bN$, let $X_\ell$ be the number of increasing paths of length $\ell$ in $(G,\pi)$. Similarly, for $v\in V$, let $Y_\ell(v)$ be the number of increasing paths of length $\ell$ starting at $v$. Lastly, for $u,v  \in V$, denote by $Z_\ell(u,v)$ the number of increasing paths of length $\ell$ starting at $u$ and ending at $v$. 

\begin{lemma}
\label{lemma:nb_paths}
For all $u,v \in V$ and for all $\ell \in \bN$,
\begin{equation*}
\begin{split}
\bE[X_\ell] &\leq  \frac{n (\log n)^\ell }{\ell!} \times (\boldsymbol{a}^T \cdot \Lambda^\ell \cdot \boldsymbol{1}), \\
\bE[Y_\ell(u)] &\leq \frac{(\log n)^\ell }{\ell!} \times (b_{\sigma(u)}^T \cdot \Lambda^\ell \cdot \boldsymbol{1}),\\
\bE[Z_\ell(u,v)] &\leq \frac{(\log n)^\ell }{a_{\sigma(v)}n\ell!} \times (b_{\sigma(u)}^T \cdot \Lambda^\ell \cdot b_{\sigma(v)}),
\end{split}
\end{equation*}
where $(b_i)_{i \in [k]}$ is the canonical basis of $\mathbb{R}^k$ and $\boldsymbol{1} = (1,\ldots,1)^T \in \mathbb{R}^k$. 
\end{lemma}

\begin{proof}
We first prove the third identity, since the other two can be easily derived from it. Fix $u,v\in V$ and denote by $Z^*_\ell(u,v)$ the number of paths (not necessarily increasing) of $\ell$ edges starting at $u$ and ending in $v$ in $G$. Notice that, since $G$ and $\pi$ are independent, for all $\ell \in \bN$,
$$\bE[Z_\ell(u,v)] = \frac{1}{\ell!} \bE[Z^*_\ell(u,v)]. $$
Hence, it suffices to show that
\begin{equation}
\label{ineq:induction}
\bE[Z^*_\ell(u,v) ]\leq \frac{(\log n)^\ell }{a_{\sigma(v)}n} \times (b_{\sigma(u)}^T \cdot \Lambda^\ell \cdot b_{\sigma(v)}). \end{equation}
We proceed by induction on $\ell \in \bN$. In the case $\ell =1$, the expected number of paths starting from $u$ and ending in $v$ is simply the probability of the edge $uv$ being present in the graph, that is,
$$\bE[Z^*_1(u,v)] = \frac{\lambda_{\sigma(u)\sigma(v)}\log n}{n} = \frac{\log n}{a_{\sigma(v)}n}\times (b^T_{\sigma(u)}\cdot \Lambda\cdot b_{\sigma(v)}). $$
Assuming that (\ref{ineq:induction}) holds for $\ell \in \bN$, we compute the expectation of $Z^*_{\ell+1}(u,v)$ by counting possible extensions of paths of length $\ell$ by an edge. For $z\in V\backslash\{u,v\}$, denote by $Z^*_\ell(u,z;v)$ the number of paths of length $\ell$ from $u$ to $z$, not containing $v$. Observe that $Z^*_{\ell}(u,z;v)$ and $Z^*_1(z,v)$ are independent. Thus,
\begin{equation*}
\begin{split}
    \bE[Z^*_{\ell+1}(u,v)] &= \sum_{z\in V\backslash\{u,v\}}\bE\left[  Z^*_\ell(u,z;v)  \right]\bE[Z^*_1(z,v)] \leq \sum_{z\in V}\bE\left[  Z^*_\ell(u,z)  \right]\bE[Z^*_1(z,v)]\\
    &\leq \sum_{z\in V} \frac{(\log n)^{\ell} }{a_{\sigma(z)}n} \times (b_{\sigma(u)}^T \cdot \Lambda^\ell \cdot b_{\sigma(z)}) \times \frac{ \Lambda_{\sigma(z)\sigma(v)}\log n}{a_{\sigma(v)}n}\\
    &= \frac{(\log n)^{\ell+1}}{a_{\sigma(v)}n} \sum_{i=1}^k \frac{|V_i|}{a_in} \times (\Lambda^\ell)_{\sigma(u)i} \Lambda_{i\sigma(v)} \\
    &= \frac{(\log n)^{\ell+1} }{a_{\sigma(v)}n} \times (b_{\sigma(u)}^T \cdot \Lambda^{\ell+1} \cdot b_{\sigma(v)}), 
\end{split}
\end{equation*}
where we used the induction hypothesis in the first inequality. This completes the proof of the upper bound on $\bE[Z_\ell(u,v)]$. Now, observe that
\begin{equation*}
    \begin{split}
        \bE[Y_\ell(u)] &\leq \sum_{v\in V} \bE[Z_\ell(u,v)] \leq \frac{(\log n)^\ell}{\ell!}\sum_{i=1}^k \frac{|V_i|}{a_i n} (\Lambda_{\sigma(u)i}) = \frac{(\log n)^\ell }{\ell!} \times (b_{\sigma(u)}^T \cdot \Lambda^\ell \cdot \boldsymbol{1}) , \\
        \bE[X_\ell] &\leq \sum_{u \in V}\bE[Y_{\ell}(u)] \leq \frac{(\log n)^\ell}{\ell!} \left(\sum_{i=1}^k |V_i|b_i \cdot \Lambda^\ell \cdot \boldsymbol{1} \right)  = \frac{n (\log n)^\ell }{\ell!} \times (\boldsymbol{a}^T \cdot \Lambda^\ell \cdot \boldsymbol{1}),
    \end{split}
\end{equation*}
which concludes the proof.
\end{proof}

The following inequalities will be used throughout this section. For any non-negative vectors $w_1,w_2 \in \mathbb{R}^k$ and $\ell = \ell(n) \to \infty$, we have that
\begin{equation}
\label{gelfands}
w_1 \cdot \Lambda^\ell \cdot w_2 \leq \|w_1 \| \|\Lambda^\ell \| \| w_2\| \leq \|w_1\| \theta^\ell \|w_2\|,
\end{equation}
where we applied Gelfand's formula (see Theorem 5.6.9, \cite{HJ85}). Furthermore, there exists $C>0$ such that
\begin{equation}
\label{exponential}
\|\textrm{exp}(\Lambda)\| \leq C e^\theta, 
\end{equation}
which can be found in \cite{J04}, Equation 2.12.

\begin{proof}[Proof of the upper bounds in Theorem \ref{thm_reachability} and Proposition \ref{proposition_subcritical}]
In the case $\theta > 1$, there is nothing to prove since the size of the reachable set cannot exceed $n$. Suppose $\theta < 1$ and let $v \in V_i$. Recall that $w \in B_{(G,\pi)}(v)$ if and only if there is an increasing path from $v$ to $w$. Hence, by the union bound, for any $N\in \bN$,
$$\bP[|B_{(G,\pi)}(v)| \geq N] \leq \bP\left[\sum_{\ell \geq 1} Y_\ell(v) \geq N\right] \leq \frac{1}{N}\sum_{\ell \geq 1} \bE[Y_\ell(v)] \leq \frac{1}{N} b^T_{\sigma(v)} \cdot e^{\Lambda\log n}\cdot \boldsymbol{1},$$
where we used Markov's inequality and Lemma \ref{lemma:nb_paths}. Using the inequality from (\ref{exponential}) yields
$$\bP[|B_{(G,\pi)}(v)|\geq N] \leq \frac{1}{N}\|b_{\sigma(v)}\|\times \|\textrm{exp}(\Lambda \log n)\|\times\|\boldsymbol{1}\| \leq \frac{C\sqrt{k}e^{\theta \log n}}{N}. $$
Setting $N = e^{(1+\epsilon)\theta \log n}$ yields the upper bound of Theorem \ref{thm_reachability}\ref{subcritical_reachability}. The lower bound in Proposition \ref{proposition_subcritical} is derived using the same arguments, replacing $\log n$ by the general $\omega_n$. \end{proof}

\begin{proof}[Proof of the upper bounds for longest paths in Theorem \ref{thm_longest_paths}]\hfill \\
We start with \ref{path}. For any $\ell \in \bN$,
$$\bP[|\gamma_{\max}|\geq \ell] \leq \bP[X_\ell \geq 1] \leq \bE[X_\ell], $$
by Markov's inequality. Hence, by (\ref{gelfands}) and Stirling's formula,
$$\bE[X_\ell] \leq \frac{n(en \theta\log n)^\ell}{\ell^\ell}, $$
so setting $\ell = \lceil a \log n\rceil$ for any $a > \alpha(\theta)$, we obtain
$$\bP[|\gamma_{\max}|\geq a \log n] \leq n^{1+a\log \frac{e\theta}{a}}= o(1), $$
since, by definition of $\alpha(\theta)$, $1+a\log\frac{e\theta}{a}<0 $. We prove \ref{path_v} the same way. Fix $v \in V$. By Markov's inequality,
$$\bP[|\gamma_{\max}(v)|\geq \ell] \leq \bP[Y_\ell(v) \geq 1] \leq \bE[Y_\ell(v)] \leq \sqrt{k} \left(\frac{e\theta\log n}{\ell}\right)^\ell = o(1),$$
when $\ell = \lceil (1+\epsilon)e\theta\log n \rceil$ for some small $\epsilon > 0$. Finally, we show \ref{path_uv}. Suppose $\theta > 1$ and let $\ell = \lceil (1+\epsilon)\beta(\theta)\log n \rceil$. Then, by the union bound and Markov's inequality,
$$\bP[|\gamma(u,v)|\geq \ell] \leq \sum_{h\geq \ell} \bE[Z_h] = O\left( \frac{1}{n}\sum_{h\geq \ell} \frac{(\theta \log n)^h}{h!} \right), $$
where the last inequality follows from Lemma \ref{lemma:nb_paths} and \eqref{gelfands}. Observe that
$$\frac{(\theta\log n)^{h+1}/(h+1)!}{(\theta\log n)^h/h!} = \frac{\theta\log n}{h+1} \leq \frac{\theta\log n}{\ell}\sim\frac{\theta}{(1+\epsilon)\beta(\theta)}<1. $$
Hence,
$$\frac{1}{n}\sum_{h\geq \ell} \frac{(\theta \log n)^h}{h!} = O\left( \frac{(\theta \log n)^\ell}{n \ell!} \right) . $$
By Stirling's formula and the definition of $\beta(\theta)$,
$$\frac{(\theta \log n)^\ell}{n \ell!} \leq n^{-1}\left(\frac{e\theta \log n }{\ell} \right)^\ell = n^{-1} \left(\frac{e\theta}{a_n} \right)^{a_n \log n} = n^{a_n\log\left(\frac{e\theta}{a_n}\right)-1} = o(1),  $$
where we set $a_n = \ell / \log n > \beta(\theta)$. We deduce that for all $\epsilon > 0$,
$$\bP[|\gamma(u,v)|\geq (1+\epsilon)\beta(\theta)\log n]=o(1), $$
which concludes the proof. However, when $\theta <1$, by union bound and Markov's inequality,
$$\bP[|\gamma_{\max}(u,v)|\geq 1]\leq \sum_{\ell\geq 1}\bP[Z_\ell(u,v)\geq 1] \leq \sum_{\ell \geq 1} \bE[Z_\ell(u,v)] =  b^T_{\sigma(u)} \cdot \left(\frac{1}{n}\sum_{\ell\geq 1} \frac{(\Lambda \log n)^\ell}{\ell!}  \right) \cdot b_{\sigma(v)}. $$
Observe that the sum above is the exponential of the matrix $\Lambda \log n$. Now, we use (\ref{exponential}) to deduce that
$$\bP[|\gamma_{\max}(u,v)|\geq1] \leq \frac{1}{n}\|b_{\sigma(u)}\|\times\|\textrm{exp}(\Lambda \log n)\|\times\|b_{\sigma(v)}\| \leq Cn^{\theta-1} = o(1),$$
since $\theta < 1$. \end{proof}

\begin{proof}[Proof of the lower bounds on shortest paths in Theorem \ref{thm_longest_paths}]\hfill \\
By Markov's inequality, we have for any $\ell \in \bN$,
$$ \bP[|\gamma_{\min}(u,v)| \leq \ell] = \bP\left[\sum_{h=1}^\ell Z_h(u,v) \geq 1 \right] \leq \sum_{h=1}^\ell \bE[Z_h(u,v)].$$
Using Gelfand's formula, for any $\epsilon > 0$, there exists $C=C(\epsilon)$ such that
$$ |b_i \cdot \Lambda^h \cdot b_j |\leq \|\Lambda^h\|_2 \leq C (\theta + \epsilon)^h.$$
Thus, it follows from Lemma \ref{lemma:nb_paths} and Stirling's formula that
$$ \bP[|\gamma_{\min}(u,v)| \leq \ell] \leq \frac{C}{a_{\max} n} \sum_{h=1}^{\ell} \frac{((\theta+\epsilon)\log n)^h}{h!} \leq \frac{C}{a_{\max} n} \sum_{h=1}^{\ell} \left(\frac{(\theta+\epsilon)e\log n}{h}\right)^h.$$
If $\ell \leq (1-\epsilon)\gamma(\theta+\epsilon)\log n$, then each term of the sum may be bounded by
$$ \left(\frac{(\theta+\epsilon)e\log n}{h}\right)^h \leq \left(\frac{(\theta+\epsilon)e}{(1-\epsilon)\gamma(\theta+\epsilon)}\right)^{(1-\epsilon)\gamma(\theta+\epsilon)\log n} = n^{(1-\epsilon)\gamma(\theta+\epsilon)\log\frac{(\theta+\epsilon)e}{(1-\epsilon)\gamma(\theta+\epsilon)}},$$
since the left-hand side is an increasing function of $h$ for $h \leq \gamma(\theta+\epsilon)\log n$. Now, since the function $x\mapsto x\log((\theta+\epsilon)e/x)$ is strictly convex and increasing for $x \leq \gamma(\theta+\epsilon)$, 
$$ (1-\epsilon)\gamma(\theta+\epsilon)\log\frac{(\theta+\epsilon)e}{(1-\epsilon)\gamma(\theta+\epsilon)} < \gamma(\theta+\epsilon)\log\frac{(\theta+\epsilon)e}{\gamma(\theta+\epsilon)},$$
and, thus, there exists $\epsilon^* > 0$ such that 
$$ \bP[|\gamma_{\min}(u,v)| \leq (1-\epsilon)\gamma(\theta+\epsilon)\log n] = O(n^{-\epsilon^*}),$$
where we applied Markov's inequality. Since this is true for any choice of $\epsilon$, we can deduce the desired result.
\end{proof}

We additionally bound the size of the set $B_{(G,\pi)}(v,S)$ of vertices $u$ reached by $v$ via increasing paths with the constraint that they pass through a given subset $S\subseteq V$, not containing $v$. 
\begin{lemma}
\label{upper:S}
Let $v\in V$ and $S\subseteq V\backslash \{v\}$. Then, whp $|B_{(G,\pi)}(v,S)| = O(|S|\log n \textrm{ }n^{\theta(1+\epsilon)})$.  
\end{lemma}
\begin{proof}
Let $Y_\ell(v,S)$ be the number of increasing paths of length $\ell$ from $v$, passing through the set $S$ at least once. An immediate adaptation of the arguments in Lemma \ref{lemma:nb_paths} yields
$$ \bE[Y_\ell(v,S)] \leq \frac{|S|\ell}{n} \times \frac{(\log n)^\ell }{\ell!} \times (b_{\sigma(v)}^T \cdot \Lambda^\ell \cdot \boldsymbol{1}),$$
since there are $\ell$ positions on the path that may contain a vertex from $S$, for which there are $|S|$ possible vertices, instead of the previous at most $n$. Hence, since $|B_{(G,\pi)}(v,S)| \leq \sum_{\ell \geq 1} Y_\ell(v,S)$, by proceeding in the same manner as in the proof of the upper bounds in Theorem \ref{thm_reachability}, we deduce the desired result.
\end{proof}

\section{Reachability from a vertex}
\label{section_subcritical}

In order to prove the lower bounds in Theorem \ref{thm_reachability}, we couple the temporal stochastic block model with temporal random multigraphs whose edge labels are exponentially distributed; these models are formally defined in Section~\ref{TESBM}. There are many benefits of working with exponentially distributed labels, namely
\begin{itemize}
	\item memoryless property of the exponential distribution;
	\item easy computation of minimum of independent exponential random variables;
	\item abundance of concentration inequalities for sums of exponential random variables.
\end{itemize}  
We begin by presenting a recursive algorithm for building reachable sets in a generalized setting (one that covers both the TSBM and the auxiliary model mentioned above). This process mirrors the one used in previous works (e.g., \cite{CRRZ24,BKL24}), where the reachable set is built by recursively adding a new vertex via the edge of minimal label, exceeding the one added at the previous step. We refine this procedure by considering a broader notion of the candidate set at each step of the exploration. This allows us to work with collections of edges whose size and composition evolve in a more regular manner over time, which in turn facilitates the probabilistic analysis. We conclude Section \ref{section:construction_process} by introducing the notion of temporal multigraphs.

In Section \ref{TESBM}, we define a temporal random stochastic block model with exponentially distributed edge labels. We state our results on reachable sets in this auxiliary setting, although their proof is detailed in Section \ref{section:proof_key_thm}. Using the construction process from Section \ref{section:construction_process}, we establish a coupling of the original temporal stochastic block model and this auxiliary process. This allows us to deduce the desired lower bounds in the subcritical regime $\theta < 1$. We use layering arguments to obtain the corresponding lower bound on the size of the reachable set in the supercritical regime $\theta > 1$. Finally, we sketch the main ingredients of the proof of Theorem \ref{thm_longest_paths}.

\subsection{Construction of reachable sets}
\label{section:construction_process}

Let $H=(V,E)$ be a simple graph and $\psi : E \to \bR^+$ be an injection. Then, $(H,\psi)$ is a temporal graph. Fix a vertex $v \in V$. We define the following process which constructs a spanning subgraph of the induced temporal graph of $B_{(H,\psi)}(v)$ in $(H,\psi)$ by consecutively adding the edge with minimal label adjacent to the already constructed graph. Note that this is essentially the construction process of the \emph{first foremost tree} (see \cite{CRRZ24}), the only difference being that here, we allow some cycles to form. In our setting, at time $t\in\bN$, $I_t$ denotes the set of vertices added by time $t$, $e_t$ is the edge added at time $t$ and $Z_t = \psi(e_t)$ denotes its label. We initialize the process with $I_0 = \{v\}$ and $Z_0 = 0$. At step $t+1$, let 
$$A_t \coloneqq \{e \in E : e \cap I_{t} \neq \emptyset,\textrm{ }  \psi(e) > Z_{t} \}$$ 
be the set of \emph{admissible} edges. Note that an edge is admissible even if \emph{both} of its endpoints are in $I_t$. If $A_t = \emptyset$, the process comes to a halt. If not, define 
$$e_{t+1} \coloneqq \argmin_{e\in A_t} \psi(e),$$ 
and denote its label by $Z_{t+1} = \psi(e_{t+1})$. Update the vertex set $I_{t+1} = I_t \cup e_{t+1}$. The constructed temporal graph at step $t$ is, thus, $\cI_t \coloneqq ((I_t,\{e_s:s\leq t\}),\psi)$
.  Denote by $\tau_A \coloneqq \inf\{t\geq 0 : A_t = \emptyset \}$. The process is, thus, stopped at time $\tau_A$, and the set of reachable vertices from $v$, $B_{(H,\psi)}(v) = I_{\tau_A}$. We often drop the $\tau_A$ notation and write simply $\cI$ and $I$ instead of $\cI_{\tau_A}$ and $I_{\tau_A}$, respectively. \\ 

A difficulty in the analysis of this process is that the set $A_t$ depends on the underlying graph, and in the random setting (in particular for the TSBM), its size and composition are not directly accessible, as they depend on the degrees of the vertices. This makes it difficult to control the minimum $Z_t$. It is more convenient to work instead with minima taken over all unexplored edges, so we extend $\psi$ accordingly. Suppose that the map $\psi$ is extended to non-edges in such a way that for all $u,v \in V$, such that $uv\notin E$, $\psi(uv) > \max_{e\in E} \psi(e)$. For $x \in \bR^+$, denote by
$$\cE(H,\psi,x) = \{\{u,w\} \subset V: \psi(uw) > x\} $$
the set of edges that arrive after time $x$. Then we may define for $t\in \bN$ a set of \emph{unexposed} edges within the process
$$U_t =  \{\{u,w\} \in \cE(H,\psi,Z_t) :  u\neq w, \{u,w\} \cap I_{t} \neq \emptyset \},$$
that is the set of pairs of distinct vertices of at least one endpoint in $I_t$, whose associated label is larger than $Z_t$. Further, define $\tau_U \coloneqq \inf\{t\geq 0 : U_t \cap E = \emptyset \}$. It is easy to check that, since $A_t = U_t \cap E$, for all $t$,
$$ \min\{\psi(uw) : \{u,w\} \in U_t \} = \min\{\psi(e) : e\in A_t\} = Z_t, $$
and $\tau_U = \tau_A \eqqcolon \tau$. Hence, considering $U_t$ rather than $A_t$ allows for an extension of the definition of $Z_t$ beyond $\tau$. Furthermore, minima are taken over all unexplored edges, so no information on vertex degrees is explicitly required in the analysis of the process. We call the sequence $(U_t)_{t\geq 1}$ of sets of pairs of vertices the \emph{canonical sequence} associated to $(H,\psi)$ and $v$. Canonical sequences play a vital role in the coupling of the temporal stochastic block model with the auxiliary temporal multigraph in the following section.  \\

It is still challenging to understand the structure and cardinality of the elements of the canonical sequence in the context of temporal stochastic block models, where we need access to the number of different types of edges of $U_t$. To facilitate the analysis further, we introduce sequences of sets of unexplored edges that evolve in a more regular manner over time. Roughly speaking, these sets are chosen so that their size is a linear function of time, and their composition preserves the proportions of the different types of edges in the stochastic block model setting. In particular, given $(H,\psi)$ and a vertex $v$, we replace the associated canonical sequence in the above described procedure by a different \emph{smoother} sequence of sets of pairs of vertices. Formally, let ${\boldsymbol{F}} = (F_t)_{t\geq 1}$ be a sequence of subsets of unordered pairs of $V$, and denote by $\cI_t^{\boldsymbol{F}}, I^{\boldsymbol{F}}_t, \tau^{\boldsymbol{F}}, e^{\boldsymbol{F}}_t, Z^{\boldsymbol{F}}_t$ the constructed graph at step $t$, its vertex set, the total number of steps, the edge we attach at step $t$ and its label, respectively. By convention, $Z^{\boldsymbol{F}}_0 = 0$. We say that $\boldsymbol{F}$ is \emph{admissible} if, for all $t\leq \tau^{\boldsymbol{F}}$,
$$F_t \subseteq \cE(H,\psi,Z^{\boldsymbol{F}}_t) \cap (I^{\boldsymbol{F}}_t \times V) ,$$
that is $F_t$ is a set of unexplored edges of at least one endpoint in $I^{\boldsymbol{F}}_t$, and whose label is larger than $Z^{\boldsymbol{F}}_{t}$. This is equivalent to having all vertices of $\cI^{\boldsymbol{F}}$ be reachable from $v$, that is $I^{\boldsymbol{F}} \subseteq B_{(H,\psi)}(v)$. Moreover, $\boldsymbol{F}$ is said to be \emph{regular} of \emph{rate} $m\in \bN$ if it is admissible and, for all $t \leq \tau^{\boldsymbol{F}}$, $|F_t| / t = m$. By working with regular sequences, we always know exactly how many edges are considered at each step of the exploration (here $tm$ at time $t$), and ensure a smooth dependence on time. Finally, for a given partition $\boldsymbol{V} = (V_1,\ldots, V_k)$ of the vertices of $H$, we say that $\boldsymbol{F}$ is \emph{block-regular} with respect to $\boldsymbol{V}$ of \emph{rate} $m\in \bN$ if it is regular of rate $m$, and additionally satisfies for all $i \in [k]$, $t \leq \tau^{\boldsymbol{F}}$, 
$$|F_t \cap (I^{\boldsymbol{F}}_t \times V_i) |/|I^{\boldsymbol{F}}_t| \in \{\lfloor m|V_i|/|V| \rfloor, \lceil m|V_i| /|V| \rceil\}.$$
Block-regularity ensures that the proportions of considered edges of each type remain unchanged during the process, that is the number of edges of $F_t$ that have an endpoint in $V_i$, is proportional to $|V_i|$. Controlling the distribution of candidate edges across blocks is essential in the temporal stochastic block model setting. Our goal is to approximate the canonical sequence of temporal stochastic block models by admissible block-regular sequences that construct a subgraph of $\cI$, spanning all but a negligible proportion of its set of vertices. Finding such sequences becomes considerably easier in a multigraph setting, where the set of candidates is much larger and the dependence on the past insignificant.  

\paragraph{Temporal multigraphs.}Let $H$ be a multigraph with no loops. The definition of $\psi$ is slightly more involved when working with multigraphs, since the number of edges indexed by $u,w \in V, u\neq w$, is \` a priori unbounded. In this case, the set of edges of $H$ is a multiset, that we embed into the set $V^2\times{\bN}$, which contains countably many copies of each unordered pair of distinct vertices $((uv)_{i})_{i\in \bN}$: 
$$E(H)\subseteq V^2\times \bN \coloneqq \{(uw)_i : u,w\in V, i\in \bN\},$$ 
and we define 
$$\psi : V^2\times\bN \to \bR^+,$$ 
in such a way that for all $u,w \in V, u\neq w$, the sequence $(\psi((uw)_i))_{i\in \bN}$ is increasing and unbounded. Since for a given pair of vertices, the labels of the possible multiedges between them is ordered, we may only consider one at a time. For any $x\in \bR^+$, denote by
$$\cE(H,\psi, t) = \{(uw)_i: u,w \in V, i = \min\{j\in \bN :  \psi((uw)_j) > t \} \} \subseteq V^2$$ 
the set containing exactly one edge per vertex pair, which is the first one to arrive after time $x$. Unlike the simple graph case, where the number of edges arriving after a given time $t$ heavily depends on $t$ and the structure of the underlying graph, here the set $\cE(H,\psi,t)$ has size $n(n-1)/2$ for all $t\in \mathbb{N}$ and any choice of $(H,\psi)$ on $n$ vertices. Define the \emph{canonical sequence} of sets of unexposed edges, which we use to construct the reachable set from a vertex $v$, by
$$ U_t = \{(uw)_i \in \cE(H,\psi,Z_t) : \{u,w\}\cap I_t \neq \emptyset\},$$
that is the set containing for each vertex pair of at least one element of $I_t$, the edge of minimal label exceeding $Z_t$. Hence, $U_t \simeq I_t \times V \subseteq V^2$, so the notions of admissible, regular and block-regular sequences of sets of unexposed edges translate in the setting of multigraphs just the same. \\

\subsection{Temporal graphs with exponentially distributed edge labels}
\label{TESBM}
We now introduce a temporal random multigraph $(G^*,\pi^*)$, defined as follows. The vertex set $V(G^*) = V(G) = [n]$, in the same way as in $G$, is partitioned into $k$ blocks. To each pair of distinct vertices $u,v\in [n]$, such that $u \in V_i, v\in V_j$, associate an increasing sequence of random labels $\pi^*((uv)_1) < \pi^*((uv)_2) < \cdots$, such that $\pi^*((uv)_\ell) \sim \textrm{Exp}(p_{ij})$ for all $\ell \in \bN$. By convention, consider that $\textrm{Exp}(0)$ is always $\infty$. Equivalently, the labels $\pi((uv)_i)$ may be interpreted as the arrival times of a Poisson process $\Pi(uv)$ of rate $p_{ij}$. The edge multiset is defined by
$$E(G^*) \coloneqq \{(uv)_i : u,v \in [n], u\neq v, i\in\bN, \pi^*((uv)_i)<1 \}. $$
Hence, $uv$ is an edge of multiplicity $\max\{i\in \bN:\pi^*((uv)_i)<1\}$. Furthermore, the probability that $uv \in V_i \times V_j$ is an edge at least once is given by 
$$\bP[\textrm{Exp}(p_{ij})<1] = 1-e^{-p_{ij}}.$$ 
Hence, $G^*$ is distributed as a multigraph version of $\bS\bB\bM$ of parameters $n,k,\boldsymbol{a}, Q\in [0,1]^{k\times k}$, defined by $Q_{ij} = 1-e^{-p_{ij}}$, since the edge probabilities are independent and constant per block pair. Observe that, $\pi^*$ is defined in such a way that all non-edges also have labels. \\

Fix a vertex $v\in [n]$. We construct $B_{(G,\pi)}(v)$ and $B_{(G^*,\pi^*)}(v)$ via the construction process described in Section \ref{section:construction_process}, using unstarred notation for quantites associated with the canonical sequence of unexposed edges of $(G,\pi)$ and $v$ (namely, $\cI_t$, $Z_t$, $e_t$, etc.), and starred notation ($\cI^*_t$, $Z^*_t$, $e^*_t$, etc.) for the corresponding quantities related to $(G^*,\pi^*)$ and $v$. By using these processes, we couple $(G,\pi)$ and $(G^*,\pi^*)$ in such a way that $B_{(G,\pi)}(v) = B_{(G^*,\pi^*)}(v)$ with high probability whenever $\theta < 1$, and thus, work only with $(G^*,\pi^*)$ in the subcritical regime. 

The following lemma gives an upper bound on the number of extra steps by step $t$, denoted by $E^*_t$, in constructing $B_{(G^*,\pi^*)}(v)$, in which we add an edge, but do not add a vertex. By controlling this quantity, we are able to relate the size of the infection with the stopping time of the process, since for all $t\leq \tau^*$, $|I^*_t| = t - E^*_t$. 

\begin{lemma}
\label{no_fail}
Suppose $\theta < 1$ and for $\epsilon > 0$, such that $(1+\epsilon)\theta < 1,$ let $r^+ = n^{\theta(1+\epsilon)}$. Then, whp, for all $t\leq \tau^* \land r^+$, $E^*_{t} = O_p(t^2/n) = o_p(t)$. In particular, when $\theta < 1/2$, choosing $\epsilon$, such that $(1+\epsilon)\theta < 1/2$, we obtain for $t\leq \tau^* \land r^+$, $E_t^* = 0$ whp.
\end{lemma}
\begin{proof}
At the beginning of the process, there are no edges, so $E^*_0 = 0$. Fix $t \geq 1$. We add an extra edge at step $t$ if $e^*_{t} \subseteq I^*_{t-1}$. Define
\begin{equation*}
    \begin{split}
        Z^{\mathrm{in}}_{t} &\coloneqq \min\{\pi^*(e)-Z^*_{t-1}: e\cap I^*_{t-1} =e \},\\
        Z^{\mathrm{out}}_{t} &\coloneqq \min\{\pi^*(e)-Z^*_{t-1}: |e\cap I^*_{t-1}|=1 \}.
    \end{split}
\end{equation*}
Then, $Z^{\mathrm{in}}_{t}$ is the minimum of at most $|I^*_{t-1}|^2$ independent exponential random variables of parameter at most $p_{\max}$, and is, thus, stochastically dominated by $\mathrm{Exp}(|I^*_{t-1}|^2 p_{\max})$. Meanwhile, since $P$ is assumed to be irreducible, for all $i\in[k]$, there exists $j \in [k]$ such that $p_{ij}\neq 0$. Hence, if we denote by $a_{\min}=\min_{i\in[k]}a_i $ and $p_{\min}=\min\{p_{ij}\neq 0, i,j\in[k]\}$, for each $u \in I^*_{t-1}$, there are at least $a_{\min}n$ vertices $w$, such that $uw \in E(G^*)$ with probability at least $p_{\min}$. Thus, $Z^{\mathrm{out}}_{t}$ is the minimum of at least $a_{\min}|I^*_{t-1}|n$ independent exponential random variables of parameter at least $p_{\min}$. Therefore, it stochastically dominates $\mathrm{Exp}(a_{\min}|I^*_{t-1}|np_{\min})$. Since for any $m\geq 1$, conditionally on $|I^*_{t-1}|=m$, $Z^{\mathrm{in}}_t+Z^*_{t-1}$ and $Z^{\mathrm{out}}_t+Z^*_{t-1}$ are independent,
$$\bP[Z_{t}^{\mathrm{in}}\leq Z^{\mathrm{out}}_{t}| |I^*_{t-1}|=m] \leq \bP[\mathrm{Exp}(m^2p_{\max})\leq \mathrm{Exp}(a_{\min}mn p_{\min}) ] = O(m/n),$$
so we deduce that
\begin{equation*}
\bP[e^*_{t}\subseteq I^*_{t-1}]=\sum_{m=1}^{t-1} \bP[|I^*_{t-1}|=m] \bP[Z_{t+1}^{\mathrm{in}}\leq Z^{\mathrm{out}}_{t+1}||I^*_{t-1}|=m] = O(t/n).
\end{equation*}
 Hence, we can bound the expectation
$$ \bE[E^*_t] = \sum_{s=1}^t \bP[e^*_s \subseteq I^*_{s-1}] = O\left( \sum_{s=1}^t  \frac{s}{n}\right) = O(t^2/n),$$
and deduce the desired results by Markov's inequality.
\end{proof}

The following theorem is the main ingredient of the proof of our main results. The proof can be found in Section \ref{section:proof_key_thm}.

\begin{theorem}
\label{key:theorem}
Let $\boldsymbol{F}=(F_t)_{t\geq 1}$ be a block-regular sequence with respect to $(V_1,\ldots,V_k)$, of rate $n_* = (1-o(1))n$, associated to $(G^*,\pi^*)$ and $v$. If $\cI^{\boldsymbol{F}}$ is a tree, then, whp
\begin{itemize}
	\item $\tau^{\boldsymbol{F}} \geq n^{\theta(1-\epsilon)}$, for any $\epsilon > 0$,
	\item for all $i\in [k]$, $t = t(n) \leq  \tau^{\boldsymbol{F}},$ $|I^{\boldsymbol{F}}_t\cap V_i|\sim_p t(\theta x_\theta)_i$,
	\item $\cI^{\boldsymbol{F}}$ is distributed as a weighted random recursive tree of size $\tau^{\boldsymbol{F}}$, and random weights $(W_m)_{m\geq 1}$, defined in the same way as in Proposition \ref{proposition_subcritical_structure}.
\end{itemize}
\end{theorem}

\begin{remark}\label{remark:preliminary_lower_bounds} Such a sequence $\boldsymbol{F}$ exists for all $\theta < 1$, see \eqref{embed_sequence} in Section \ref{proof_lower_bounds} below for an explicit construction. Since $\boldsymbol{F}$ is block-regular and so, in particular, admissible, it follows that $\tau^* \geq \tau^{\boldsymbol{F}}$ whp, and, moreover for all $i \in [k],$ $t\leq \tau^{\boldsymbol{F}}$, we have 
$$|I^*_t \cap V_i| \geq |I^{\boldsymbol{F}}_t \cap V_i| = \Omega_p(t).$$
\end{remark}

\subsection{Coupling the TSBM to temporal random multigraphs}
\label{subsection_multi}

Let $G \sim \bS\bB\bM(n,k,\boldsymbol{a},P)$ and $(\psi(e))_{e\in E(G)}$ be a random permutation of the edges. Since we are only interested in the ordering of the labels of the edges, it is equivalent to assign to each edge $e \in E(G)$ an independent label $\pi(e) \sim \mathrm{Unif}[0,1]$ and consider $(G,\pi)$ (see \cite{CRRZ24}). Moreover, without loss of generality, we may couple $G$ and $\pi$, by extending the definition of $\pi$ to all unordered pairs $uv$ as independent random variables $\pi(uv) \sim \textrm{Unif}[0,1/p_{\sigma(u)\sigma(v)}]$, and declaring $uv \in E(G)$ if and only if $\pi(uv) < 1$, where recall that for $w\in [n]$, $\sigma(w)$ indicates which block $w$ belongs to. By convention, consider that $\textrm{Unif}[0,1/0]$ is always $\infty$. This coupling is introduced for technical convenience and simplifies our arguments. \\

Now, we are ready to couple the temporal stochastic block model $(G,\pi)$ with $(G^*,\pi^*)$. 

\begin{lemma}
\label{lemma:coupling_G}
If $\theta<1$, then there exists a coupling of $(G^*,\pi^*)$ and $(G,\pi)$, such that $\cI = \cI^*$ with high probability.
\end{lemma}

\begin{remark}\label{result:multiTSBM} Note that, the event $\cI = \cI^*$ implies that $\tau = \tau^*$, and that $\cI^*$ is a simple graph. 
\end{remark}

Before proving Lemma \ref{lemma:coupling_G}, we state the following technical result, whose proof can be found in Appendix \ref{proof:claim}.

\begin{claim}
\label{claim:calcul_Y}
Fix $z \in [0,1),$ $a\geq 2$ integer, $p \in (0,1]$. Let $(X_i)_{i\in [a]}$ be a family of i.i.d. uniform random variables on the interval $[0,1/p]$, and let $(Y_i)_{i \in [a]}$ be i.i.d. exponential random variables of parameter $p$. Define
\begin{equation*}
X = \min_{i\in [a]}X_i \indicatrice_{X_i > z}  , \quad \quad Y = \min_{i\in [a]}Y_i \indicatrice_{Y_i > z}.
\end{equation*}
Then, there exists a coupling $\bP_{(X,Y)}$, such that
$$\bP_{(X,Y)}[X\neq Y] \leq \frac{1}{a} + e^{-a(1-p)} + p. $$	
\end{claim}

\begin{proof}[Proof of Lemma \ref{lemma:coupling_G}]
For $t\in \bN$, define the events 
$$\cZ_t = \{\forall s \leq t, Z_s = Z^*_s\},\quad \textrm{and} \quad \cE_t = \{\forall s \leq t, e_s = e^*_s\}.$$
Observe that $(\cZ_t)_t$ and $(\cE_t)_t$ are both decreasing. Furthermore, $\{\cI_t = \cI^*_t \} \subseteq \cZ_t \cap \cE_t$, and $\{\tau = \tau^*\} \subseteq \cZ_{\tau \land \tau^* +1}$. Indeed, it follows from the fact that
\begin{equation*}
\tau = \inf \{t\in \bN : Z_t > 1\}, \quad \quad \tau^* = \inf\{t\in \bN : Z^*_t > 1 \}.
\end{equation*} 
Hence, it is sufficient to show that there exists a coupling $\bP_{*}$, such that
$$\bP_*[\cZ_{\tau \land \tau^* +1} \cap \cE_{\tau^*}] =  1-o(1).$$ 
Now, for $i,j \in [k],$ $i\leq j$ and $t\in \bN$, denote by 
$$U_t(i,j) = U_t \cap (V_i \times V_j), \quad \quad U^*_t(i,j) = U^*_t \cap (V_i \times V_j).$$ 
Define
$$ Y_t(i,j)= Y_t(j,i) = \inf_{ \{u,w\} \in U_t(i,j)}\pi(uw),\quad \quad Y^*_t(i,j) = Y^*_t(j,i) = \inf_{\{u,w\} \in U_t^*(i,j)}\pi^*(uw),$$
where, by convention, $\inf \emptyset = \infty$. Finally, define the event 
$$\cY_t = \{\forall s\leq t, \forall i,j \in [k], Y_s(i,j)=Y_s^*(i,j)\}.$$
Observe that $\cZ_t \subseteq \cY_t$. When $t=0$, with probability 1, $\cI_0 = \cI^*_0$. Fix $i,j \in [k]$. Notice that, conditionally on $\cF_t \coloneqq \sigma(\cY_s,\cI_s,e_s$ for all $s\leq t$), $U_t(i,j) \subseteq U^*_t(i,j)$ for all $i,j \in [k]$. Denote by $\wt{Y}^*_t(i,j) = \inf_{\{uw\} \in U_t(i,j)} \pi^*(uw)$. Then, notice that if $e_t \notin V_i\times V_j$, then $Y_t(i,j) = Y_{t-1}(i,j) =Y^*_{t-1}(i,j) = Y^*_{t}(i,j)$. Otherwise, $Y_t(i,j)$ is the minimum of $|U_t(i,j)|$ i.i.d. uniform random variables on $[0,1/p_{ij}]$ conditioned on being greater than $Z_{t-1}$, whereas $\wt{Y}^*_t(i,j)$ is the minimum of $|U_t(i,j)|$ i.i.d. exponential random variables of parameter $p_{ij}$, conditioned on being greater than $Z_{t-1}$. Crucially, in this case, both $Y_t(i,j)$ and $\wt{Y}_t^*(i,j)$ are independent of the past. Hence, we can apply Claim \ref{claim:calcul_Y} by setting $X = \wt{Y}_t^*(i,j)$, $Y=Y_t(i,j)$, $a = |U_t(i,j)|$, $p=p_{ij}$ and $z=Z_{t-1}$. We deduce that $Y_t(i,j)$ and $\wt{Y}_t^*(i,j)$ can be coupled in such a way that
$$ \bP[Y_t(i,j) \neq \wt{Y}^*_t(i,j) | \cY_t,\cI_t=\cI^*_t] \leq \frac{1}{|U_t(i,j)|}+ e^{-|U_t(i,j)|(1-p_{ij})}+p_{ij} = O(\log n/n),$$
since $|U_t(i,j)| = \Omega(n)$ for all $t=o(n)$. By the union bound over all possible choices of $i,j$ (less than $k^2$) and then over all times up to $\tau+1 = o(n)$ (by the established upper bound in Theorem \ref{thm_reachability}\ref{subcritical_reachability}), we deduce that $\cY_{\tau+1}$ holds with high probability, which implies that $\tau = \tau^*$. It remains to show that $e_t=e^*_t$ for all $t\leq \tau$, with high probability. Fix $t \leq \tau$ and denote by $(i_0,j_0) = \argmin_{(i,j)\in [k]^2} Y_t(i,j)$. Then, notice that any edge from $|U^*_t(i_0,j_0)|$ has equal probability to be $e^*_t$, since their labels are i.i.d. Similarly, $e_t$ is a uniformly random edge from $|U_t(i_0,j_0)|$. Thus, we can couple the two in such a way that 
$$e_t = \wt{e}^*_t \coloneqq \argmin_{uw \in U_t(i_0,j_0)} \pi^*(uw).$$ 
Hence, if we denote by $\cU_t = \sigma(U_t(i_0,j_0),U_t^*(i_0,j_0))$,
$$\bP[e_t\neq e^*_t|\cU_{t}] = \bP[Y_t^*(i_0,j_0) \neq \wt{Y}^*_t(i_0,j_0)|\cU_{t}] = \frac{|U^*_t(i_0,j_0)|-|U_t(i_0,j_0)|}{|U^*_t(i_0,j_0)|}.$$
Now, observe that $|U^*_t(i_0,j_0)|-|U_t(i_0,j_0)| \leq t+ t\textrm{ deg}_{\max}(G)$. 
This follows directly from 
$$ U_t^* = U^*_{t-1} \cup ((I^*_t\backslash I^*_{t-1}) \times (I^*_t)^c ), \quad \quad U_t = (U_{t-1} \backslash \{e_t\})\cup \{uw\in (I_t\backslash I_{t-1})\times I_t^c: \pi(uw)>Z_{t-1}\},$$
since there are $t$ constructed edges by time $t$, and at each step, if $e_t \not\subseteq I_t$, that is we add a new vertex $v_t = I_t\backslash I_{t-1}$, then we only consider those edges between $v_t$ and $I^c_t$ that have labels greater than $Z_{t-1}$, which does not exceed 1 for all $t \leq \tau\land \tau^*$. Now, since by standard arguments (namely Chernoff bound and the union bound), $\textrm{deg}_{\max}(G) = O_p(\log n)$, and since it follows from Remark \ref{remark:preliminary_lower_bounds} that $|U^*_t(i_0,j_0)| \geq a_{i_0} (n-1) |I^*_t \cap V_{j_0}| = \Omega_p(tn)$,
$$\bP[e_t\neq e^*_t|\cU_{t}]= \frac{|U^*_t(i_0,j_0)|-|U_t(i_0,j_0)|}{|U^*_t(i_0,j_0)|} = O_p\left(\frac{\log n}{n} \right).$$
Therefore, putting everything together yields that
$$ \bP[\cZ_{\tau \land \tau^* +1} \cap \cE_{\tau^*}] = O\left( n^{(1+\epsilon)\theta} \times \frac{\log n}{n}\right), $$
where we used that for all $\epsilon > 0$, $\tau < n^{(1+\epsilon)\theta}$ whp, which follows from the upper bound in Theorem \ref{thm_reachability}\ref{subcritical_reachability}. Choosing $\epsilon$ such that $(1+\epsilon)\theta < 1$, allows us to conclude.
\end{proof}

\subsection{Proof of Proposition \ref{proposition_subcritical_structure} and the lower bounds in Theorem \ref{thm_reachability}}
\label{proof_lower_bounds}

The proof of the lower bounds in Proposition \ref{proposition_subcritical} is an easier version of the proof of the lower bounds in the critical regime $\omega_n = \log n$, so we omit the former. 

\paragraph{Subcritical regime $\theta < 1$.} It follows from Lemma \ref{lemma:coupling_G} that it is sufficient to prove our main results for $(G^*,\pi^*)$. We do so by \emph{approximating} the canonical sequence of $(G^*,\pi^*)$ by a block-regular one. Namely, it is sufficient to construct a sequence $\boldsymbol{F} = (F_t)_{t\geq 1}$ of sets of pairs of vertices that is block-regular with respect to $(V_1,\ldots,V_k)$, of rate $n_* = n-o(n)$, such that $\cI^{\boldsymbol{F}}$ is a tree, and $\tau^{\boldsymbol{F}} \sim_p \tau^*$. Indeed, we can deduce the lower bound of the size of the reachable set and the embedding of the weighted random recursive tree immediately from Theorem \ref{key:theorem}. For the proportions of reachable vertices per type, it is sufficient to show that
$$ |B_{(G^*,\pi^*)} \cap V_i|  \sim_p |I^{\boldsymbol{F}}\cap V_i|. $$
A direct consequence of the upper bound from Theorem \ref{thm_reachability}\ref{subcritical_reachability} and Remark \ref{result:multiTSBM} is that for any $\epsilon>0$, $\tau^* = \tau < n^{\theta(1+\epsilon)}$. Hence, it follows from Lemma \ref{no_fail} that the extra steps $E^*_{\tau^*}=o(\tau^*)$. Remark \ref{remark:preliminary_lower_bounds} implies that
$$ |B_{(G^*,\pi^*)} \cap V_i| \sim_p \tau^* - \sum_{j\neq i} |B_{(G^*,\pi^*)} \cap V_j| \lesssim \tau^* - \tau^{\boldsymbol{F}} + |I^{\boldsymbol{F}}\cap V_i| \sim_p |I^{\boldsymbol{F}}\cap V_i|,$$
as desired. We proceed by showing a construction of $\boldsymbol{F}$.

Let $S_1$ be an arbitrary set of $k \tau^*$ vertices, such that $|S_1 \cap V_i| = \tau^*$ for all $i\in [k]$. Set $F_1 = U^*_1 \backslash \{uw: u \in S_1\}$. Having constructed $S_{t-1}$ and $F_{t-1}$, denote by $v^{\boldsymbol{F}}_{t-1}$ the only element of $I^{\boldsymbol{F}}_{t-1}\backslash I_{t-2}^{\boldsymbol{F}}$ (which is never, since no extra edges are allowed), and let
$$ S_t = S_{t-1} \backslash \{w_t \},$$
where $w_t$ is an arbitrary vertex of $S_{t-1}$, satisfying $\sigma(w_t) = \sigma(v^{\boldsymbol{F}}_{t-1})$. We may always find such a vertex, since for all $i\in[k]$, 
$$|S_{t-1}\cap V_i| \geq |S_1 \cap V_i| - (t-1) \geq \tau^* - \tau^{\boldsymbol{F}} + 1 \geq 1.$$ 
Set
\begin{equation}
\label{embed_sequence}
F_t = F_{t-1} \backslash (I_{t-1}\times \{v^{\boldsymbol{F}}_{t-1}\}) \cup (I_{t-1}\times \{w_t\}).
\end{equation}
It is straightforward to check that $\boldsymbol{F}$ is block-regular of rate $n_* = n- k\tau^* = (1-o(1))n$. Moreover, by construction, no extra edges are allowed, that is $\cI^{\boldsymbol{F}}$ is a tree and, hence, $v^{\boldsymbol{F}}_t$ is well-defined. Finally, $I^* - I^{\boldsymbol{F}} \subseteq B_{(G^*,\pi^*)}(v,S_1)$, hence it follows from Lemma \ref{upper:S} that $|I^*| \sim_p |I^{\boldsymbol{F}}| = \tau^{\boldsymbol{F}}$. Since $\tau^* \sim_p |I^*|$ by Lemma \ref{no_fail}, we deduce that $\tau^{\boldsymbol{F}} \sim_p \tau^*$, as desired. 

\paragraph{Special case $\theta < 1/2$.} The only result left to prove in this regime is the one from Proposition \ref{proposition_subcritical_structure}, namely that the induced graph $G[B_{(G^*,\pi^*)}(v)]$ is with high probability a tree distributed as a weighted random recursive tree. It is sufficient to show that there exists a sequence $\boldsymbol{J}$ of sets of pairs of vertices that satisfies all the assumptions of Theorem \ref{key:theorem}, and is \emph{equivalent} to the canonical sequence $(U^*_t)_{t\geq 1}$, in the sense that $\cI^* = \cI^{\boldsymbol{J}}$. Consider the sequence $\boldsymbol{J}$, where, for each $t \geq 1$, $J_t$ contains all elements of $U^*_t$ with additional independent copies of all $uw \in U^*_t \cap (I^*_t \times I^*_t)$. In other words, $\boldsymbol{J}$ is the canonical sequence associated to a \emph{directed} version of $(G^*,\pi^*)$, where $\pi^*$ now associates an increasing sequence of exponential random variables to each \emph{ordered} pair of vertices, independent of the others, and where the notion of reachability is further constrained to directed paths. Hence, only the extra edges may be oriented in both directions, and $\boldsymbol{J}$ is, thus, regular of rate $n-1$. Crucially, at each step $t\geq 1$, there are exactly twice as many candidates for extra edges, that is,
$$|J_t\cap (I^{*}_{t-1}\times I^*_{t-1})| = 2|U^*_t \cap (I^*_{t-1}\times I^*_{t-1})|,$$
so Lemma \ref{no_fail} can be easily adapted to this setting. We deduce that, with high probability $(J_t \backslash U^*_t)\cap E(G^*) = \emptyset$ and, thus, $\cI^{\boldsymbol{J}}=\cI^*$ is a tree with high probability. It remains to show that $G[B_{(G^*,\pi^*)}(v)] = \cI^*$ with high probability. Conditionally on $\cI^*$ being a tree, denote by $v_t^*$ the vertex added at step $t$. Then, for all $t <\tau^*$,
$$\bP[\exists u\in I^*_{t-1}: \pi^*((uv^*_t)_1)<Z_t] \leq (t-1)p_{\max}, $$
by the union bound, and the fact that $Z_t < 1$. Hence, the expected number of such edges up to time $\tau^*$ is of order at most $\log n(\tau^*)^{2}/n = o(1)$, since $\theta < 1/2$, and the upper bound from Theorem \ref{thm_reachability}\ref{subcritical_reachability} applies to $\tau^*$, see Remark \ref{remark:preliminary_lower_bounds}. We conclude by Markov's inequality.

\paragraph{Supercritical regime $\theta > 1$.} We show that when $\theta>1$, almost all vertices may be reached from $v$. We do so by \emph{layering} a finite number of subcritical TSBMs on top of each other. This argument is formalized by the following proposition.

\begin{proposition}\label{layering}
Let $m \in \bN$ be a fixed integer and let $(G_i)_{i\in [m]}$ be a sequence of $\bS\bB\bM\left(n,k,\boldsymbol{a},\frac{1}{m}P\right)$. Then, we can couple $(G,\pi)$ with $(G_i)_{i\in[m]}$ in such a way that 
$$E(G) = \bigsqcup_{i=1}^{m} E(G_i), \quad \textrm{ and } \quad \forall i\in [m-1],\quad \max_{e \in E(G_i)} \pi(e) < \min_{e\in E(G_{i+1}) }\pi(e).$$
\end{proposition}

\begin{proof}
To each edge $e \in E(G)$, associate a uniform random variable $W_e$ on $[m]$, independent of the others. Then, set $E(G_i) = \{e\in E(G):W_e = i\}$. Moreover, let $\pi_{|_i}$ be a uniform permutation of $E(G_i)$, independent of the others. Then, for $e \in E(G)$, set 
$$\pi(e) = \sum_{i=1}^{m} \indicatrice_{W_e=i} \left(\sum_{j=1}^{i-1}\max_{e^\prime \in E(G_{j})}\pi_{|_{j}}(e^\prime) + \pi_{|_i}(e) \right), $$
where, by convention, $\pi_{|_0} \equiv 0$. This coupling yields the correct marginals for $\pi$ and $E(G_i)$, for all $i$. Indeed, for all distinct $u,v \in V(G), i\in [m]$,
$$ \bP[uv \in E(G_i)] = \bP[uv \in E(G)]\bP[W_{uv} = i|uv \in E(G)] = p_{\sigma(u)\sigma(v)} / m,$$
independently of the other pairs of vertices. Hence, $G_i$ is distributed as a $\bS\bB\bM$ of probability matrix $P/m$, as desired. It remains to check whether $\pi$ is a uniform permutation of the edges. Observe that, it is, by definition, the concatenation of $m$ uniform permutations on the random partition $\sqcup_{i\in[m]}E(G_i)$ of $E(G)$. Hence, if all such partitions are equiprobable and exchangeable, $\pi$ will be a uniform permutation itself. Since the random variables $\{W_e, e\in E(G)\}$ are i.i.d., we may conclude.
\end{proof}

Consider a decomposition of the supercritical TSBM $(G,\pi)$ into $2\lfloor \theta \rfloor$ subcritical ones, of probability matrix $P/(2\lfloor \theta \rfloor)$, that we denote by $(G_i,\pi_i)$, for $i \in [2\lfloor \theta \rfloor]$. Construct the reachable set of $v$ in $G_1$, that we denote by $B_1(v)$. Then, for each $w \in [n] \backslash B_1(v)$, construct $B_2(w)$ the reachable set via decreasing paths of $w$ in $G_2$. The results from Theorem \ref{thm_reachability}\ref{subcritical_reachability} apply to $B_2(w)$, since reversing the temporal order turns decreasing paths into increasing paths and does not change the distribution of the model. In fact, the lower bounds on the size of the reachable set in the subcritical regime hold with probability at least $1-n^{-\varepsilon}$, for some positive $\varepsilon > 0$. Thus, with high probability at least $(1-o(1))n$ of the vertices from $[n]\backslash B_1(v)$ reach at least $n^{1/2+\delta}$ vertices via decreasing paths in $G_2$, for some $\delta >0$, since $\theta/(2\lfloor \theta \rfloor) > 1/2$. Denote this set by $S$. In the same way, the number of vertices reached from $v$ via increasing paths in $G_1$, is $|B_1(v)| \geq n^{1/2 + \delta}$. Since, by construction, all edges of $G_1$ arrive before those of $G_2$, for all $w \notin B_1(v)$,
$$w \in B_{(G,\pi)}(v) \quad \textrm{ if and only if } \quad B_1(v)\cap B_2(w) \neq \emptyset. $$
This condition is satisfied for all elements of $S$ with high probability, by the birthday paradox. Indeed, for any $i \in [k]$ and $w\in S$, both $B_1(v)\cap V_i$ and $B_2(w)\cap V_i$ are uniformly random subsets of $V_i$ of size $\Omega_p(n^{1/2+\delta})$. Hence,
$$\bP\left[B_1(v)\cap B_2(w)\cap V_i = \emptyset\bigg{|}\min\{B_1(v)\cap V_i,B_2(w)\cap V_i\}= m\right] \leq \left( 1-\frac{m}{a_i n} \right)^{m} \leq e^{-m^{2}/n}, $$
so setting $m = n^{1/2+\delta}$ yields the desired result. Since $|S| = (1-o(1))n$ and $S \subseteq B_{(G,\pi)}(v)$ with high probability, we may conclude.
\subsection{Proof of Theorem \ref{thm_longest_paths}}
\label{section:longest_paths}

We start with the lower bounds in \ref{path} and \ref{path_v}. In the case $\theta<1$, the longest increasing path of a fixed endpoint $v$ is naturally lower bounded by the height of the weighted random recursive tree we embed in it, which is, with high probability, $e\theta \log n$ (see Theorem 5, \cite{S18}). In the supercritical regime $\theta > 1$, it is straightforward to see that we can concatenate $\lceil \theta \rceil$ paths of length $(\theta e \log n )/ \lceil \theta \rceil$ using the layering arguments from Proposition \ref{layering}. \\

To prove the remaining bounds in \ref{path_uv}, it is sufficient to show that we can construct an increasing path from $u$ to $v$ of length $x \log n$, for any $x \in (\gamma(\theta),\beta(\theta))$. The proof is identical to the one in Section 5 from \cite{BKL24}, provided we show an equivalent of their Lemma 15, where $P$ plays the role of $p$, namely that, in a temporal stochastic block model of parameter $P/2$, there exists $\delta > 0$, such that, with high probability, there are at least $n^{1/2 + \delta}$ vertices that can be reached from $v$ via increasing paths of length $\frac{x}{2}\log n \pm (\log n)^{1/2}$. In the case $\theta < 2$, this is an immediate consequence of the results on the profile of the weighted random recursive tree we embed into the reachable set, see Theorem 5 from \cite{S18}. For general $\theta \geq 2$, analogous arguments using Proposition \ref{layering} to decompose $(G,\pi)$ into subcritical TSBMs, yield the desired result.

\section{Proof of Theorem \ref{key:theorem}}
\label{section:proof_key_thm}

Consider a generalized P\'olya urn process $P=(P(n))_{n \in \bN}$, defined as follows. There are balls of $q\in\bN$ types, and the process $P(n)=(P_1(n),\ldots,P_q(n)) \in \bN^q$ keeps track of the composition of the urn at time $n$. In particular, $P_i(n)$ is the number of balls of type $i$ in the process after $n$ steps. The urn starts with a given initial condition $P(0)$, possibly random. Associate to each type $i$ an \emph{activity} $w_i \geq 0$, and a random $q$-dimensional non-negative vector $\xi_i = (\xi_{i1},\ldots,\xi_{iq})$. Let ($\xi^{(n)}_i)_{i\in [q],n\in \bN}$ be i.i.d. copies of $(\xi_i)_{i\in [q]}$. The urn evolves according to a Markov process. At each time $n\geq 1$, a ball $U_n$ is drawn from the urn at random, considering that any ball of type $i$ has weight $w_i$, that is
$$\bP[U_n \textrm{ of type } i] = \frac{w_i P_i(n-1)}{\sum_{j \in [q]}w_j P_j(n-1) } .$$
The drawn ball is then returned along with $\sum_{i\in [q]}\indicatrice_{\{U_n \textrm{ of type }i\}}\xi^{(n)}_{ij}$ balls of type $j$. In other words,
$$P(n) = P(n-1) + \sum_{i=1}^q \indicatrice_{\{U_n \textrm{ of type }i\}}\xi^{(n)}_{i} .$$
Finally, define the \emph{mean replacement matrix} $A \in \bR^{q \times q}$ of the P\' olya urn $P$ by
$A_{ij} = w_i \bE[\xi_{ij}]$ for all $i,j \in [q]$. Recall that, if $A$ is irreducible, then by the generalization of the Perron-Frobenious theorem for non-negative matrices (see Chapter 7, \cite{M23}), there exists a unique simple positive maximum eigenvalue $\vartheta>0$ of $A$, and its associated right eigenvector $\nu$ is comprised only of non-negative entries. The following result is due to Janson, and characterizes the asymptotic behavior of such P\' olya urns.

\begin{theorem}[Theorem 3.21, \cite{J04}]
\label{PU_Janson}
    If $A$ is irreducible, then as $n\to \infty,$
    $$\frac{1}{n}P(n) \toas \vartheta\nu, $$
    where $\nu$ is normalized to satisfy $(w_1,\ldots,w_k)\cdot \nu=1.$
\end{theorem}

Recall the definition of $\Lambda \in \bR^{k \times k}$ from Section \ref{introduction}, given by
$$\forall i,j\in [k], \Lambda_{ij} \coloneqq a_j \lambda_{ij},$$
and recall that $\theta > 0$ is its largest eigenvalue. In this setting, we work with $(G^*,\pi^*)$, a fixed vertex $v$, and a sequence $\boldsymbol{F}$ that is block-regular with respect to $(V_1,\ldots,V_k)$, of rate $n_* = (1-o(1))n$. Furthermore, throughout this section, $\cI^{\boldsymbol{F}}$ is conditioned on being a tree. \\

Now, we need to understand the distribution of the labels $(Z^{\boldsymbol{F}}_t)_{t \in [\tau^{\boldsymbol{F}}]}$. Recall that $\sigma(w)= i $ if $w \in V_i$. For $i \in [k]$, $t \geq 0$, denote by $X^{\boldsymbol{F}}_i(t) = |I^{\boldsymbol{F}}_t \cap V_i|$ the number of vertices of type $i$ already added to the tree by time $t$. Initially, $X^{\boldsymbol{F}}(0) = (X^{\boldsymbol{F}}_1(0),\ldots,X^{\boldsymbol{F}}_k(0)) = b_{\sigma(v)}$, where recall that $(b_i)_{i\in [k]}$ is the canonical basis of $\bR^k$. By construction, $Z^{\boldsymbol{F}}_1$ is the minimum of $n-1$ independent exponential random variables, hence 
$$Z^{\boldsymbol{F}}_1 \sim \textrm{Exp}\left(\sum_{i=1}^k a_i n_* p_{i\sigma(v)} \right) = \textrm{Exp}( (\boldsymbol{1} \cdot \Lambda \cdot X^{\boldsymbol{F}}(0))n_*\log n/n ). $$ 
Similarly, at step $t$, $Z^{\boldsymbol{F}}_t$ is the minimum of $tn_*$ exponential random variables, conditionned to be larger than $Z^{\boldsymbol{F}}_{t-1}$. By the memoryless property of the exponential distribution, 
\begin{equation}
\label{eq:Z}
Z^{\boldsymbol{F}}_t-Z^{\boldsymbol{F}}_{t-1} \sim \textrm{Exp}( ( \boldsymbol{1}\cdot \Lambda \cdot X^{\boldsymbol{F}}(t))n_*\log n/n ).
\end{equation}
In order to understand the distribution of the edge labels $(Z^{\boldsymbol{F}}_t)_t$, we need to first analyze the process $(X^{\boldsymbol{F}}(t))_t$. We do so by coupling it to the following P\' olya urn process. Let $P=(P(t))_{t\in\bN}$ be a P\' olya urn of $k$ types, activities $w \in \bR^k$, given by $w_i = \sum_{\ell \in [k]} \Lambda_{i\ell} $, and mean replacement matrix $\Lambda^T$. It readily follows from Theorem \ref{PU_Janson} that
\begin{equation}
\label{PU:limit}
\frac{1}{t}P(t) \toas \theta x_\theta, \quad \textrm{as } t\to \infty. 
\end{equation}

\begin{lemma}
\label{PU}
There exists a coupling of $(X^{\boldsymbol{F}}(t))_t$ and $(P(t))_t$, such that for all $t\leq \tau^{\boldsymbol{F}}$,
$$P(t)= X^{\boldsymbol{F}}(t).$$
\end{lemma}

\begin{proof}
Recall that, throughout this section, $\cI^{\boldsymbol{F}}$ is conditioned on being a tree. Denote by $u_t \in I^{\boldsymbol{F}}_{t-1}$ and $v_t \in V\backslash I^{\boldsymbol{F}}_{t-1}$ the endpoints of $e^{\boldsymbol{F}}_t$. We proceed by induction. Here, the balls of the urn at time $t$ are identified with the constructed set $I^{\boldsymbol{F}}_t$. A ball of type $i$ corresponds to a vertex from $V_i$. The initial condition coincides by design. For $t \geq 1$, set $U_t= u_t$. Thus, $U_t$ is of type $i \in [k]$ if $U_t \in V_i$. Furthermore, conditionally on $U_t\in V_i$, set for all $j \in [k]$,
$$\xi^{(t)}_{ij} = \indicatrice_{v_t \in V_j}, $$
which is a non-negative $k$-vector. This coupling yields the correct marginal distribution of $(P(t))_t$. Indeed, for all $t\geq 1$, first observe that 
$$e^{\boldsymbol{F}}_{t+1} \in (I^{\boldsymbol{F}}_{t}\cap V_i)\times V_j \Longleftrightarrow \min_{(u,v)\in F_t\cap \left((I^{\boldsymbol{F}}_{t}\cap V_i)\times V_j\right)} \pi^{\boldsymbol{F}}(u,v) < \min_{(u,v)\in F_t \backslash \left((I^{\boldsymbol{F}}_{t}\cap V_i)\times V_j\right)} \pi^{\boldsymbol{F}}(u,v),  $$
so, it follows from (\ref{eq:Z}) that
$$\bP[e^{\boldsymbol{F}}_{t+1} \in (I^{\boldsymbol{F}}_{t}\cap V_i)\times V_j|{\cI}^{\boldsymbol{F}}_{t}] = \frac{b_j\cdot \Lambda \cdot X^{\boldsymbol{F}}_i(t)b_i }{\boldsymbol{1}\cdot \Lambda \cdot X^{\boldsymbol{F}}(t)} = \frac{\Lambda_{ji}X^{\boldsymbol{F}}_i(t)}{\boldsymbol{1}\cdot \Lambda \cdot X^{\boldsymbol{F}}(t)},  $$
where we used the fact that for any two independent exponential random variables of parameters $x,y \in \bR^+$, $\bP[\textrm{Exp}(x)<\textrm{Exp}(y)] = x/(x+y)$ and $\min(\textrm{Exp}(x),\textrm{Exp}(y)) \sim \textrm{Exp}(x+y)$. Hence,
\begin{equation*}
    \bP[U_{t+1}\in V_i|{\cI}^{\boldsymbol{F}}_t] = \sum_{j=1}^k \bP[e^{\boldsymbol{F}}_{t+1} \in (I^{\boldsymbol{F}}_{t}\cap V_i)\times V_j|{\cI}^{\boldsymbol{F}}_{t}] = \frac{\boldsymbol{1}\cdot \Lambda \cdot X^{\boldsymbol{F}}_i(t)b_i }{\boldsymbol{1}\cdot \Lambda \cdot X^{\boldsymbol{F}}(t)} = \frac{w_iX^{\boldsymbol{F}}_i(t)}{\sum_{j=1}^k w_j X^{\boldsymbol{F}}_j(t)},
\end{equation*}
as desired, assuming $P(t)=X^{\boldsymbol{F}}(t)$. Now, conditionally on $U_{t+1} \in V_i$, for all $j\in [k]$, we add a ball of type $j$ with probability
$$\bE[\xi^{(t+1)}_{ij}|{\cI}^{\boldsymbol{F}}_{t}] = \bP[v_{t+1} \in V_j | u_{t+1} \in  V_i] = \frac{\bP[e^{\boldsymbol{F}}_{t+1} \in (I^{\boldsymbol{F}}_{t}\cap V_i)\times V_j|{\cI}^{\boldsymbol{F}}_{t}]}{\bP[U_{t+1}\in V_i|{\cI}^{\boldsymbol{F}}_t]}=\frac{\Lambda_{ji}}{w_i}, $$
so the mean replacement matrix is, indeed, $\Lambda^T$. This concludes the proof.
\end{proof}

\begin{corollary}
Conditionally on $\cI^{\boldsymbol{F}}$ being a tree, for all $t\leq \tau^{\boldsymbol{F}}$, the tree of vertex set $I_t^{\boldsymbol{F}}$ and edges $\{e^{\boldsymbol{F}}_s: s\leq t\}$ is distributed as a weighted random recursive tree of size $t$, and random weights defined in the same way as in Proposition \ref{proposition_subcritical_structure}.
\end{corollary}

We can now analyze the stopping time $\tau^{\boldsymbol{F}}$.

\begin{lemma}\label{concentration_labels}
Let $\epsilon > 0$ and set $r^- = n^{(1-\epsilon)\theta}$. Then, whp $Z^{\boldsymbol{F}}_{r^-} < 1$. In particular, $\tau^{\boldsymbol{F}} \geq r^-$. \end{lemma}

Before we prove this result, we need the following technical lemma, due to Janson. The proof can be found in \cite{J18}.

\begin{lemma}
\label{janson1}
Let $r \in \mathbb{N}$, $(\lambda_i)_{i\in [r]} \in (\mathbb{R}^+)^r$ and $X_i \sim \mathrm{Exp}(\lambda_i)$ be independent random variables. Define
$$ \mu = \mathbb{E}\left[ \sum_{i=1}^r X_i \right] = \sum_{i=1}^r \frac{1}{\lambda_i}, \quad \quad \lambda_* = \min_{i \in [r]} \lambda_i. $$
Then, for $\epsilon >0$, we have 
$$ \mathbb{P}\left[ \sum_{i=1}^r X_i \geq (1+\epsilon)\mu \right] \leq (1+\epsilon)^{-1} e^{-\lambda_* \mu (\epsilon-\log(1+\epsilon))}.$$
\end{lemma}

\begin{proof}[Proof of Lemma \ref{concentration_labels}]
Denote by $\Delta Z^{\boldsymbol{F}}_t \coloneqq Z^{\boldsymbol{F}}_t - Z^{\boldsymbol{F}}_{t-1}$ with the convention $Z^{\boldsymbol{F}}_0 = 0$. Let $z_t \coloneqq (\boldsymbol{1}\cdot \Lambda \cdot X^{\boldsymbol{F}}(t))/t$. From \eqref{eq:Z}, $\Delta Z^{\boldsymbol{F}}_t \sim \textrm{Exp}(z_t t\log n)$. Crucially, conditionally on the process $(X^{\boldsymbol{F}}(s))_{s\leq t}$, $(\Delta Z^{\boldsymbol{F}}_s)_{s\leq t}$ are mutually independent random variables. It follows from Lemma \ref{PU} that $z_t \overset{\bP}{\to} \theta$, that is, whp, for all $\epsilon^*>0$, there exists $M>0$, such that for all $t\geq M$, $|z_t-\theta|<\epsilon^*$. Denote by 
$$z_{\min} = \min_{t\leq r^-} z_t \geq \min_{i \in [k]}\sum_{j=1}^k \Lambda_{ij} > 0, $$
since $\Lambda$ is assumed to be irreducible. Hence, we can construct a sequence of mutually independent exponential random variables $X_i$ of parameters $\lambda_i = z_{\min}i\log n$ for $i < M$, and $\lambda_i = (\theta-\epsilon^*)i\log n$ for $i \geq M$ such that $\Delta Z^{\boldsymbol{F}}_i$ is stochastically dominated by $X_i$, for all $i \leq r^-$. Notice that $\lambda_* = \min_{i\in [r]}\lambda_i \to +\infty$ as $n\to \infty$. Moreover,
\begin{equation*}
\begin{split}
\mu = \bE\left[\sum_{i=1}^{r^-} X_i\right] &= \sum_{i=1}^M\left(\frac{1}{z_{\min}i\log n} -\frac{1}{(\theta-\epsilon^*)i\log n} \right) + \sum_{i=1}^{r^-} \frac{1}{(\theta-\epsilon^*)i\log n}\\ 
&= O(M/\log n)+\frac{\log r^-}{(\theta-\epsilon^*)\log n} = \frac{(1-\epsilon)\theta}{\theta-\epsilon^*}\pm o(1)<1,
\end{split}
\end{equation*}
where the last inequality is made possible by choosing $\epsilon^*$ small enough. Lemma \ref{janson1} yields
$$\bP[Z^{\boldsymbol{F}}_{r^-} \geq 1] \leq \bP\left[\sum_{i=1}^{r^-} X_i \geq 1\right] =O(n^{-z_{\min}\varphi(\epsilon)}) =o(1), $$
where $\varphi(\epsilon) > 0$ is some positive constant, depending only on $\epsilon$. 
\end{proof}

\section{Discussion}
\label{discussion}

\subsection{Variations of the model}
Here, we list some related models of the temporal stochastic block model for which our results hold as well.

\paragraph{Multigraphs.} Multigraphs are a natural extension of the model, where repetitions of edges are allowed. For fixed $m \in \bN$, sample $m$ random edges from the complete graph independently, where each edge $e \in V_i\times V_j$ has probability proportional to $p_{ij}$ to be selected at each step. Denote the resulting temporal graph by $H_m = H_m(n,k,\boldsymbol{a},P)$, where each edge is assigned a uniformly random label. In the case where $p_{ij}\equiv p$ for all $i,j \in [k]$, this is known as the \emph{random gossiping protocol}, considered in \cite{BS79,H81,M72,MSA20,MSRA16,VKS17}. Our results transfer to this model for an appropriate parametrization of the parameter $m$, namely $m \sim \sum_{i,j \in [k]} a_ia_jp_{ij}n^2$, the average number of edges in the simple temporal stochastic block model. Indeed, by conditioning $H_{m}$ on being a simple graph, we can couple $(G,\pi)$ and $H_m$, in such a way that $(G,\pi)\subseteq H_{m}$. This allows us to deduce all lower bounds. Conversely, all of our first-moment arguments for non-existence of increasing paths can be adapted to the model $H_{m}$.

\paragraph{Directed graphs.} Another variation of the temporal stochastic block model is one where edges are directed, that is, the probability matrix $P$ is not necessarily \emph{symmetric}. Directed stochastic block models (see, e.g., \cite{WW87, A15, AB00}) naturally model infections, where, for instance, one community may adopt protective measures (e.g., wearing masks), leading to asymmetric transmissions across communities. With the additional constraint that $u$ is reached by $v$ if there exists and increasing \emph{directed} path from $u$ to $v$, it is easy to see that all our results regarding reachability extend to this setting, since we never use the symmetry of $P$ in our analysis.

\paragraph{Random blocks.} The classical stochastic block model is usually defined in a slightly different way than in this paper. Instead of using the vector $\boldsymbol{a} = (a_1,\ldots,a_k)$ as the deterministic proportion of each block, vertices are assigned a community at random according to probabilities given by $\boldsymbol{a}$. It is straightforward to show that the two models are equivalent when examining their asymptotic behavior, the main ingredient of the proof being the law of large numbers.

\subsection{Related open problems}

We conclude with several natural directions for further work.

\begin{enumerate}
    \item \emph{The critical window.}
    We do not study the behavior of the reachable set at the critical point $\theta=1$. Describing the size and structure of $B_{(G,\pi)}(v)$ in the critical window remains open, even in the temporal Erd\H os--R\'enyi case.
    \item \emph{Unbounded number of communities.}
    Throughout the paper we assume that $k$ is fixed. A natural extension is to allow $k=k(n)\to\infty$. In the limit, this leads to temporal versions of inhomogeneous random graphs or graphon models, where connection probabilities may vary across pairs of vertices.

    \item \emph{Inhomogeneous order of connection probabilities.}
    Assumption \ref{ass2} requires all nonzero connection probabilities to be of the same order. It would be interesting to understand what happens when different blocks interact at different rates, especially in sparse or heavy-tailed regimes.
    
    \item \emph{Blocks of sublinear size.}
    Throughout the paper, all communities have linear size. It would be interesting to understand what happens when some blocks are sublinear. Such blocks need not be negligible: already in the case $k=2$, the expression in Example \ref{example_explicit} suggests that interactions between two blocks may contribute to the leading eigenvalue when $|V_i|\times|V_j|\times p_{ij}$ is of sufficiently large order, so the connection probabilities may \emph{compensate} the small size of a block. In this setting, the behavior of the reachable set may depend on the community containing the origin.
\end{enumerate}

\bibliography{bib}
\bibliographystyle{unsrturl}
\newpage
\appendix
\section{Proof of Claim \ref{claim:calcul_Y}}
\label{proof:claim}

Observe that $Y$ is distributed as the minimum of $a$ i.i.d. uniform random variables of distribution $\textrm{Unif}([z,1/p])$. Hence, $Y$ has density 
$$f_Y(x) = \partial_x (1-\bP[Y \geq x]) = \partial_x \left(-\left( \frac{1/p - x}{1/p - z}\right)^{a}\right)= \frac{ap(1-px)^{a-1} }{(1-pz)^{a}}, $$
for all $x \in [z,1/p]$. Meanwhile, $X$ is the minimum of $a$ i.i.d. exponential random variables of parameter $p$, conditioned on being greater than $z$. Hence, by the memoryless property, we obtain that the density of $X$ is given by
$$f_X (x) = ape^{-ap(x-z)}, $$
for all $x \geq z$. We now couple $X$ and $Y$ in such a way that
$$ \bP_{(X,Y)}[X \neq Y ] = \tv(X,Y).$$
This is indeed possible, as such a coupling exists (see \cite{L92}, Theorem I.5.2). We now compute
\begin{equation*}
\begin{split}
\tv(X,Y) &= \frac{1}{2}\int_z^\infty \left| f_Y(x)\indicatrice_{x \leq 1/p} - f_X(x)\right|dx\\
&= \frac{1}{2}\left(\int_z^{\frac{1}{p}} ap\left|\frac{(1-px)^{a-1}}{(1-pz)^a} - e^{-ap(x-z)}  \right|dx + \int_{\frac{1}{p}}^\infty ap e^{-ap(x-z)}dx\right).
\end{split}
\end{equation*}
We change the variable in both integrals $y = p(x-z)$ to obtain
\begin{equation*}
\begin{split}
\tv(X,Y) &= \frac{1}{2}\left(\int_0^{1-pz} a\left|\frac{(1-pz-y)^{a-1}}{(1-pz)^a} - e^{-ay}  \right|dy + \int_{1-pz}^\infty a e^{-ay}dy\right)\\
&\leq \frac{1}{2}\left(\int_0^{\alpha} ae^{-ay}\alpha^{-a}\left|e^{ay}(\alpha-y)^{a-1} - \alpha^a \right|dy + e^{-a(1-p)}\right),
\end{split}
\end{equation*}
where we set $\alpha \coloneqq 1-pz$, and we used that $1-pz < 1$ and $z<1$ in the second integral. Let $h(y) = e^{ay}(\alpha-y)^{a-1} - \alpha^a$. Then, by computing the derivative of $h$, $$h^\prime(y) = ae^{ay} (\alpha-y)^{a-2}(-y-1+\alpha+1/a),$$ 
we see that on the interval $(0,\alpha)$, 
$$h \textrm{ is } \begin{cases} \textrm{decreasing } &\textrm{ if } z\geq \frac{1}{ap},\\
\textrm{admits a local maximum at }\frac{1}{a}-pz &\textrm{ otherwise.}\end{cases}$$
Since $h(0) = \alpha^{a-1}(1-\alpha) > 0$ and $h(\alpha) = -1 < 0$, we deduce that, for all $z\in (0,1)$, there exists a unique solution $y_a \in (0,\alpha)$ to $h(y)=0$. Then, we split the first integral at $y_a$ to remove the absolute value 
\begin{equation*}
\begin{split}
\tv(X,Y)&\leq \frac{1}{2}\left(\int_0^{y_a} \frac{a(\alpha-y)^{a-1}}{\alpha^a} - ae^{-ay} dy + \int_{y_a}^\alpha a e^{-ay} -\frac{a(\alpha-y)^{a-1}}{\alpha^a} dy + e^{-a(1-p)}\right)\\
&= e^{-ay_a} - \alpha^{-a}(\alpha-y_a)^a - \frac{1}{2}e^{-a\alpha} +\frac{1}{2}e^{-a(1-p)} \\
&= e^{-ay_a}(1-\alpha+y_a) - \frac{1}{2}e^{a\alpha} + \frac{1}{2}e^{-a(1-p)},
\end{split}
\end{equation*}
where we used in the last equality the fact that $h(y_a)=0$ by definition. Now, since on the interval $(0,1)$, $x\mapsto xe^{-ax}$ is maximal at $x = 1/a$, and $z < 1$, it follows that
$$\tv(X,Y) \leq \frac{1}{a} + p + e^{-a(1-p)}, $$
which concludes the proof.

\end{document}